\documentclass[10pt]{amsart}
\usepackage{amssymb,amsmath,graphicx,verbatim,eufrak,amsthm}
\usepackage{color}
\usepackage{xcolor}
\usepackage{mdframed}

\usepackage{hyperref}
\hypersetup{
    colorlinks=true,       % false: boxed links; true: colored links
    linkcolor=red,          % color of internal links (change box color with linkbordercolor)
    citecolor=green,        % color of links to bibliography
    filecolor=magenta,      % color of file links
    urlcolor=cyan           % color of external links
}

\newcommand{\thmc}[1]{{\color{#1} $\bullet$}}

\newtheorem{thm}{\thmc{blue} Theorem}
\newtheorem{cor}[thm]{\thmc{cyan} Corollary}
\newtheorem{lem}[thm]{\thmc{magenta} Lemma}

\newtheorem{prop}[thm]{\thmc{cyan} Proposition}

\newtheorem{ques}[thm]{Question}

\newtheorem*{clm*}{Claim}

\theoremstyle{definition}
\newtheorem{dfn}[thm]{\thmc{red} Definition}
\newtheorem{exm1}[thm]{Example}

\theoremstyle{remark}
\newtheorem{rem}[thm]{Remark}

\newenvironment{lem*}[1]{\vspace{1ex}\noindent
{\bf Lemma* (#1).} [restatement]  \hspace{0.5em} \em }{ }
\newenvironment{thm*}[1]{\vspace{1ex}\noindent 
{\bf Theorem* (#1).} [restatement]  \hspace{0.5em} \em }{ }

{\begin{exm1}}%
{\hfill$\bigtriangleup\bigtriangledown\bigtriangleup$
\end{exm1} }

\newcommand{\abs}[1]{\left\vert#1\right\vert}

\newcommand{\set}[1]{\left\{#1\right\}}

\newcommand{\sr}[1]{\left(#1\right)}

\newcommand{\Integer}{\mathbb{Z}}

\newcommand{\Z}{\Integer}
\newcommand{\N}{\mathbb{N}}

\newcommand{\R}{\mathbb{R}}

\newcommand{\eps}{\varepsilon}

\newcommand{\ie}{{\em i.e.\ }}
\newcommand{\eg}{{\em e.g.\ }}

\DeclareMathOperator{\E}{\mathbb{E}}     % Without under-subscripts
\DeclareMathOperator{\Var}{Var}

\renewcommand{\Pr}{}
\let\Pr\relax
\DeclareMathOperator{\Pr}{\mathbb{P}}

\newcommand{\bP}{\mathbf{P}}
\newcommand{\bE}{\mathbf{E}}

\newcommand{\1}[1]{\mathbf{1}_{\set{ #1 } }}

\def\squareforqed{\hbox{\rlap{$\sqcap$}$\sqcup$}}
\def\qed{\ifmmode\squareforqed\else{\unskip\nobreak\hfil
\penalty50\hskip1em\null\nobreak\hfil\squareforqed
\parfillskip=0pt\finalhyphendemerits=0\endgraf}\fi}

\newcommand{\ignore}[1]{ }

\newcommand{\p}{\partial}

\newcommand{\dist}{\mathrm{dist}}
\newcommand{\vphi}{\varphi}

\renewcommand{\o}{\omega}

\newcommand{\Ee}{\mathcal{E}}

\newcommand{\AND}{\qquad \textrm{and} \qquad}

\renewcommand{\O}{\Omega}

\newcommand{\define}[1]{{\bf #1}}

\newcommand{\mn}{\wedge}

\usepackage{capt-of}

\newcommand{\SAW}{\mathsf{SAW}}
\newcommand{\NB}{\mathsf{NB}}

\begin{document}

\title[SAW on finite graphs]{Self-avoiding walks on finite graphs of large girth}

\author[Ariel Yadin]{Ariel Yadin$^*$}
\thanks{$^*$Ben-Gurion University of the Negev. {\em email:} \texttt{yadina@bgu.ac.il} \\
Supported by the Israel Science Foundation grant No. 1346/15. \\
Thanks to Vincent Beffara, Itai Benjamini and Hugo Duminil-Copin for useful discussions.} 

\maketitle

\begin{abstract}
We consider self-avoiding walk on finite graphs with large girth.  
%We study different aspects of this model which may be considered as 
%measuring ``mean-field'' behavior.  
%We show that large-girth graphs exhibit such behavior in some aspects but not others.
We study a few aspects of the model originally considered by Lawler, Schramm and Werner on finite balls in $\Z^d$.
The expected length of a random self avoiding path is considered.
We discuss possible definitions of ``critical'' behavior in the finite volume setting.
We also define a ``critical exponent'' $\gamma$ for sequences of graphs of size tending to infinity,
and show that $\gamma = 1$ in the large girth case.
\end{abstract}

\section{Introduction}

\subsection{Self-avoiding walks}

A \define{self-avoiding walk} is a path in a graph that does not 
visit any vertex more than once.
Counting the number of self-avoiding walks of length $n$  started at the origin in $\Z^d$,
is a long standing open problem.
It is very difficult to come up with formulas that capture the correct asymptotics.  
In fact, even the exponential growth rate of the number of such walks is difficult to precisely
calculate in most cases; this number $\mu$ is known as the {\em connective constant} of the lattice.
%Precisely, let $c_n$ denote the number of self-avoiding walks of length $n$ started at $o$ in an infinite 
%transitive graph $G$.
%It is not difficult to show that $c_{n+m} \leq c_n \cdot c_m$ so Fekete's Lemma tells us that
%the limit $\mu = \mu(G) = \lim  (c_n)^{1/n}$ exists.
%Hence the number of self avoiding-walks of length $n$ is approximately $\mu^n$.
%Determining the lower order terms in the aymptotics of $c_n$ is a major open problem.
%Even calculating the precise value of $\mu$ is usually very difficult.
%In a seminal work the connective constant of the hexagonal lattice 
%has recently been calculated by Duminil-Copin and Smirnov \cite{DS12}
%to be $\sqrt{2 + \sqrt{2}}$.  Self-avoiding walks in the plane
%admit a conjectured conformal invariant  scaling  limit, and determining the lower order terms for $c_n$
%is deeply connected to this fact.
For more on self-avoiding walks in the Euclidean context see \cite{hugo, MadrasSlade}.

Let us briefly introduce the main model usually considered in the plane (and in fact in $\Z^d$ in general).
This is sometimes called {\em Lawler-Schramm-Werner} model of self-avoiding walk \cite{LSW5}.
For some parameter $0< x \in \R$ and scaling factor $\delta>0$ consider
the finite graph $G_\delta : = \delta \Z^d \cap B(0,1)$, where $B(0,1)$ is the Euclidean ball of radius $1$.
Let $\SAW_\delta$ be the set of all self-avoiding walks in $G_\delta$ started at $0$ 
with an endpoint in the boundary of $G_\delta$ (there are finitely many such walks).
We may define a probability measure $\Pr_x^{\delta}$ 
on $\SAW_\delta$ by letting the probability of $\o \in \SAW_\delta$ 
be proportional to $x^{|\o|}$ where $|\o|$ is the length of $\o$.

It is known that the model undergoes a phase transition at $x_c = \mu^{-1}$.
Ioffe \cite{Ioffe} has shown that for $x<x_c$ the measures $\Pr_x^{\delta}$ converge 
(in an appropriate sense) to a measure on geodesics from $0$ to the boundary of $B(0,1)$.
For $x>x_c$ one may show that the limiting curve fills the ball $B(0,1)$ (again, in an appropriate sense) 
\cite{DKY12}.

The major question is to understand what happens at the critical point $x=x_c$.
In dimensions $d \geq 4$ the limiting curve is expected to be scaling to a Brownian motion;
this is related to works of Brydges \& Spencer 
\cite{BS85} and Hara \& Slade \cite{Hara08, HaraSlade1,HaraSlade2, HaraSlade3}
(see also \cite{BDS11, BEI92, BI03A, BI03B, BrydgesImbrieSlade, BS10} for the {\em upper critical dimension} $d=4$,
and the book \cite{MadrasSlade} and references therein).  This is known for $d \geq 5$ using {\em lace expansion}
\cite{HaraSlade2}. 
Dimension $d=3$ is the most mysterious.
In dimension $d=2$ the limiting curve is conjectured to be $\mathsf{SLE}_{8/3}$, Schramm-Loewner Evolution
of parameter $\tfrac83$ \cite{LSW5}.  Not much has been rigorously proven regarding 
the critical two-dimensional case; even very intuitive facts are quite involved, see \eg \cite{HugoAlan}.

\subsection{Finite graphs}

Self-avoiding walks on general graphs have received much less attention than the Euclidean lattices.
(See, for example, \cite{AlmJanson, GZ13, MadrasWu} and references therein.)
In this note we adapt the Lawler-Schramm-Werner model to the setting of {\em finite} graphs.
%We focus on 
%consider 
%self-avoiding walks on graphs of large girth.
%, and rapid enough mixing
%(which is equivalent to good enough expansion properties).
%If a $d$-regular graph has large girth, then locally around any vertex it looks like a $d$-regular tree.
%Thus, a bit of thought leads us to guess that the analogue of the connective constant should be $(d-1)$
%in this case. 
The problem on finite graphs is that there is no canonical way to define
``criticality'' and it is not clear what ``mean field behavior'' is.
This is the purpose of our definition of the critical exponent $\gamma$ and critical sequences below, 
and their relations to the expected length of a self-avoiding path and the asymptotic behavior
of the partition function.  This is explained in analogy to the more classical Euclidean space setting.

There are two main results of this paper.
The first, is the definition of the notion of critical sequences and critical exponent for the finite graph setting,
with the different viewpoints relating them to expected length and intersection of independent self-avoiding paths.
This is presented in Section \ref{scn: critical} and Theorem \ref{thm: critical and I}.

The second main result is the analysis of the critical behavior of self-avoiding walks in the large-girth case.
We have two different types of possible ``mean field'' behaviors (in analogy to the complete graph case
and the Euclidean case), and we show that large girth graphs 
only exhibit one of these (namely critical exponent $\gamma=1$).
This is done in Theorems \ref{thm: large girth} and \ref{thm: gamma}.

Let us precisely define the model and state the results.

\subsection{The model}

A {\em path} $\o$ in a graph $G$ is a sequence $(\o_0,\o_1,\ldots,\o_n)$ of vertices such that 
$\o_j \sim \o_{j+1}$ for all $j<n$, where $x \sim y$ means $x,y$ are adjacent in the graph $G$.
For such a path $\o = (\o_0,\o_1,\ldots,\o_n)$ we use $|\o| = n$ to denote the length of the path,
which is the number of edges traversed.

For a graph $G$ we denote by $\SAW(G)$ the set of finite length self avoiding walks in $G$; that is,
$$ \SAW(G) = \set{ \o \ : \ \o \textrm{ is a finite path and }  \forall \ k \neq j \  ,\  \o_k \neq \o_j } . $$
We use $\SAW(o,G)$ where $o$ is a vertex in $G$ to denote the set of all self avoiding walks in $G$ starting at $o$.
$\SAW_k(G)$ (resp.\ $\SAW_k(o,G)$) 
denotes the set of those self avoiding walks in $\SAW(G)$ (resp.\ $\SAW(o,G)$) 
which have length $k$.  

%\begin{dfn}
%Let $(G_n)_n$ be a sequence of finite graphs of sizes $|G_n| \to \infty$.
%Let $o_n \in G_n$ be a fixed vertex for every $n$.
%
%For every $n$, and a parameter $x>0$, define a probability measure on $\SAW(o_n,G_n)$ by
%setting 
%$$ Z_n(x) = Z_{o_n,G_n}(x) := \sum_{\o \in \SAW(o_n,G_n)} x^{|\o|} \AND
%\mu_{x,n} [ \o ] = \Pr_{x,o_n,G_n} [\o] = \frac{x^{|\o|} }{ Z_{n}(x) } \1{ \o \in \SAW(o_n,G_n) } . $$
%Expectation with respect to $\mu_{x,n}$ is denoted $\E_{x,n} = \E_{x,o_n,G_n}$.
%
%We dub this the {\bf self-avoiding walk}, or simply $\SAW$, on the sequence $(o_n,G_n)_n$.
%\end{dfn}

\begin{dfn}
Let $G$ be a graph (finite or infinite).  Let $o \in G$ be some vertex.

For a real parameter $x>0$ define the \define{partition function}
$$ Z_{o,G}(x) := \sum_{\o \in \SAW(o,G)} x^{|\o|} = \sum_{n=0}^{\infty} |\SAW_n(o,G)| \cdot x^n . $$
(When $G$ is infinite this converges for $x < \mu^{-1}$ where $\mu=\mu(G)$ is the connective constant of $G$.)
For any $x$ for which the partition function converges, 
define a probability measure on $\SAW(o,G)$ by
$$ \Pr_{x,o,G}[\o] = (Z_{o,G})^{-1} \cdot x^{|\o|} . $$
\end{dfn}

When $G$ is a transitive graph we will usually omit the root vertex $o$, since any vertex plays the same role;
\eg on a transitive graph $G$ by $\Pr_{x,G}$ we mean $\Pr_{x,o,G}$ for some $o$.

Expectation under $\Pr_{x,o,G}$ is denoted $\E_{x,o,G}$ (or $\E_{x,G}$ when $G$ is transitive).

We will be interested in two main quantities:
\begin{itemize}
\item The expected length of a random element under $\Pr_{x,o,G}$;
that is 
$$ L(x,o,G) = \E_{x,o,G} [|\o|] = (Z_{o,G}(x))^{-1} \cdot \sum_{n=0}^{\infty} |\SAW_n(o,G)| \cdot n \cdot x^n . $$

\item The probability that two independent samples from $\Pr_{x,o,G}$ intersect trivially; that is,
$$ I(x,o,G) = (\Pr_{x,o,G} \times \Pr_{x,o,G}) [ \o \cap \o' = \set{o} ] . $$
(Here $\o,\o'$ are independent samples each with law $\Pr_{x,o,G}$.)
\end{itemize}
Of course the analysis of the partition function $Z_{o,G}(x)$ plays an important part.
As usual, the vertex $o$ is omitted in the notation for transitive graphs.

\subsection{Critical exponents}
\label{scn: critical exp}

We start by a brief review of critical exponents in the classical case, as motivation for our definition 
of the critical exponent $\gamma$ below.

Let $G$ be some transitive infinite graph, and fix some origin $o \in G$.
The classical literature on self-avoiding walks is interested in determining the existence and values
of the so-called {\em critical exponents}.  We will not go into all the details here, see 
\cite{hugo} and \cite{MadrasSlade} for more.
One of these critical exponents is $\gamma$.  It is defined as the number such that for some constant $A$,
$$ |\SAW_n(o,G)| \sim A \cdot \mu^n \cdot n^{\gamma-1} \qquad \textrm{as } n \to \infty . $$
(Here $f(x) \sim g(x)$ as $x \to a$ means that $\lim_{x \to a} f(x)/g(x) = 1$ and the limit exists.)

For example, lace expansion methods show that when $G=\Z^d$ for $d \geq 5$ we have that the 
exponent $\gamma$ exists and $\gamma = 1$ \cite{HaraSlade1, HaraSlade2, HaraSlade3}.
This is what is called {\em mean field behavior}, because the analogous quantity for the simple random walk
is also $1$ in dimensions $d \geq 5$.

$\gamma$ also has a probabilistic interpretation.  If it exists, 
then the probability that two independent uniformly chosen self-avoiding walks from $\SAW_n(o,G)$
intersect only at the origin is $\sim c n^{1-\gamma}$ as $n \to \infty$.

%The partition function in the infinite case, also called the {\em susceptibility}, is defined as
%$$ Z_G(x) = \sum_{\o \in \SAW(o,G) } x^{|\o|} = \sum_{n=0}^{\infty} |\SAW_n(o,G)| \cdot x^n , $$
%which converges for $x < \mu^{-1}$ where $\mu = \mu(G)$ is the connective constant of $G$.
It is possible to show that the exponent $\gamma$ exists if and only if 
for some constants $c,c'$,
$$ Z_G(x) \sim c ( 1 - \mu \cdot x )^{-\gamma} \AND 
L(x,G) \sim c' (1- \mu \cdot x)^{-1} \qquad \textrm{as } x \nearrow \mu^{-1} , $$
(see \eg \cite{hugo}). 

%We see that $\gamma$ is not captured in the expected length of a random $\Pr_{x,G}$ self-avoiding walk.
%However, 
If $\gamma$ exists then $\gamma  = \lim_{x \nearrow \mu^{-1} } \frac{ \log Z_{o,G}(x) }{ \log L(x,o,G) } . $
%\begin{align}
%\label{eqn: gamma}
%\gamma & = \lim_{x \nearrow \mu^{-1} } \frac{ \log Z_{o,G}(x) }{ \log L(x,o,G) } . 
%\end{align}
Thus, with this intuition in mind we define
\begin{align}
 \gamma(x,o,G) & := \frac{ \log Z_{o,G}(x) }{ \log (L(x,o,G)+1) } . 
\end{align}

In Proposition \ref{prop:  intersection} we show that for transitive $G$,
$$ L(x,G) + 1 = I(x,G) \cdot Z_G(x) . $$
Plugging in the above predictions we have that if the exponent $\gamma$ exists then
$$ I(x,G) \sim c (1-x \mu)^{\gamma-1} . $$
Thus, mean field $\gamma=1$, implies that $I(x,G)$ 
is bounded away from $0$ as $x \nearrow \mu^{-1}$.  
Non mean field exponent, $\gamma > 1$,
implies that this probability converges to $0$ as $x \nearrow \mu^{-1}$.
%So one mean field characterization can be the property that independent self-avoiding walks 
%have positive probability of a trivial intersection, at criticality.

These properties motivate our definition of the critical exponent in the finite case.

\subsection{Criticality in the finite setting}
\label{scn: critical}

Now let us return to the finite setting.
Let $(G_n)_n$ be a sequence of %transitive 
finite graphs with $|G_n| \to \infty$, and $o_n \in G_n$ some root vertex.

\begin{dfn}
For a sequence $(x_n)_n$ of positive real numbers:
%\begin{itemize}
%\item We say that the sequence $(x_n)_n$ is \define{super-critical} (for $\SAW$ on $(o_n,G_n)_n$) if 
%$$ \limsup_{n \to \infty} I(x_n,o_n,G_n) = 0 . $$
%\item We say that the sequence $(x_n)_n$ is \define{sub-critical} if 
%$$ \liminf_{n \to \infty} I(x_n,o_n,G_n) > 0 . $$
%\item We say that the sequence $(x_n)_n$ is \define{critical} if for 
%any sequence $(\eps_n)_n$ such that $\liminf_{n \to \infty} \eps_n>0$,
%the sequence $(x_n+\eps_n)_n$ is super-critical and $(x_n-\eps_n)_n$ is sub-critical.
%\end{itemize}
%%%%%%  was the old one ^^^
\begin{itemize}
\item We say that sequence $(x_n)_n$ is \define{super-critical} (for $\SAW$ on $(o_n,G_n)_n$) if
$$ \liminf_{n \to \infty} Z_{o_n,G_n}(x_n) = \infty . $$
\item We say that the sequence $(x_n)_n$ is \define{sub-critical} if 
$$ \limsup_{n \to \infty} Z_{o_n,G_n}(x_n) < \infty . $$
\item We say that the sequence $(x_n)_n$ is \define{critical} if for 
any $0<\eps<1$ the sequence $(x_n(1+\eps))_n$ is super-critical and the sequence $(x_n(1-\eps))_n$
is sub-critical.
\end{itemize}
\end{dfn}

%The intuition behind this definition is that super-critical $\SAW$s are long, and thus
%two independent samples should have small probability of intersecting only at the origin.
%Conversely, in the sub-critical phase the walks should be short, and have a positive probability 
%of trivial intersection. 

Critical sequences are unique in the following sense.

\begin{prop} \label{prop:unique}
Let $(x_n)_n$ be a critical sequences for $(o_n,G_n)$. 
Then $(y_n)_n$ is a critical sequence if and only if 
$\tfrac{x_n}{y_n} \to 1$ as $n \to \infty$.
\end{prop}

Another result is a characterization of critical sequences in terms of the intersection probability 
function $I$.

\begin{thm}
\label{thm: critical and I}
For $\SAW$ on a sequence of transitive graphs $G_n$ with $|G_n| \nearrow \infty$,
and a sequence $(x_n)_n$:
\begin{itemize}
\item $(x_n)_n$ is super-critical if and only if $\limsup_n I(x_n,G_n) = 0$.
\item $(x_n)_n$ is sub-critical if and only if $\liminf_n I(x_n,G_n) > 0$.
\end{itemize}
\end{thm}

This theorem is not surprising, since we expect that super-critical paths should be ``long'', so that a trivial intersection
has small probability, and sub-critical paths are expected to be short, 
so there is a reasonable probability of a trivial intersection.

%We may now define the critical exponent $\gamma$ in the finite graph case.
%
%\begin{dfn}
%Let $(G_n)_n$ be a sequence of graphs with $|G_n| \nearrow \infty$.
%Let $o_n \in G_n$ be some fixed vertex.  Consider the $\SAW$ model on $G_n$.
%
%For a (critical) sequence $(x_n)_n$, define 
%$$ \gamma = \gamma((x_n,o_n,G_n)_n) = \lim_{n \to \infty} \frac{\log Z_{o_n,G_n}(x_n) }{ \log (L(x_n,o_n,G_n)+1) } . $$
%\end{dfn}

%\note{open problem: $\gamma$ is the same for critical sequences, $\gamma = 1$ for sub-critical, and $\gamma = \infty$ 
%for super-critical}

%\subsection{Expected length}
%
%As stated above, the probability of trivial intersection $I(x,o,G)$ is related on transitive graphs to 
%the expected length $L(x,o,G)$.  We will also provide asymptotics for this quantity.
%It is intuitive that $L(x_,o_n,G_n)$ should be small for sub-critical sequences and large for super-critical sequences.
%We make this precise in Theorem \ref{???} 
%and compare to the case where $G_n$ is the complete graph on $n$ vertices (Theorem \ref{???}). 
%The complete graph setting is another possible mean field scenario.  
%

\subsection{Mean field}

The main purpose of this paper is to analyze the finite graph model.
We begin with the simplest model, the complete graph.

\begin{thm}
\label{thm: mean field}
Consider the $\SAW$ model on the sequence $(G_n)_n$
where $G_n$ is the complete graph on $n$ vertices.
Then, $(\tfrac1{n})_n$ is a critical sequence.

Moreover, for all $\eps>0$: 
\begin{itemize}
\item If $(x_n)_n$ is a sequence such that 
$\lim_{n \to \infty}  (n x_n) = 1+\eps , $
then  
$$ \lim_{n \to \infty} \frac{ L(x_n,G_n)  }{ n } = \frac{\eps}{1+\eps} 
\AND \lim_{n \to \infty} I(x_n,G_n) = 0 . $$

\item If $(x_n)_n$ is a sequence such that $\lim_{n \to \infty} (n x_n) = 1-\eps$ then,
$$ \lim_{n \to \infty} L(x_n,G_n) = \frac{1-\eps}{\eps}  \AND \lim_{n \to \infty} I(x_n,G_n) = 1 . $$

\item If $x_n = \tfrac1{n}$ then there exists some (explicit) universal constant $\alpha> 0$ such that
$$ \lim_{n \to \infty} \frac{ L(x_n,G_n) }{ \sqrt{n} } = \alpha . $$
\end{itemize}
\end{thm}

%\note{add $I(x,o,G)$ in the critical case!}

\begin{thm} \label{thm: gamma MF}
Consider the $\SAW$ model on the sequence $(G_n)_n$
where $G_n$ is the complete graph on $n$ vertices.   For critical and sub-critical sequences
$(x_n)_n$ such that $\limsup_{n \to \infty} (n x_n) \leq 1$,
the exponent $\gamma$ has the mean field value
$$ \lim_{n \to \infty} \gamma(x_n,G_n) = 1 . $$

For a super-critical sequence $(n x_n) \to 1+\eps$, we have 
$$ \lim_{n \to \infty} \gamma(x_n,G_n) = \infty. $$
\end{thm}

\subsection{Euclidean case}
It is interesting to compare our findings on the complete graph to the more classical Euclidean case.

When $G_n$ is the $n \times n$ torus $(\Z/n\Z)^2$, 
the above remarks on the planar self-avoiding walk lead us to the conclusion
that 
$$ L(x,G_n) = 
\begin{cases}
O(1) & x < \mu(\Z^2)^{-1} \\
\Omega(n^2) & x > \mu(\Z^2)^{-1} \\
n^{4/3} & x = \mu(\Z^2)^{-1} ,
\end{cases}
$$
where the last value in the critical case is only conjectured,
and has to do with the fact the the Hausdorff dimension of $\mathsf{SLE}_{\kappa}$ is $1+ \tfrac{\kappa}{8}$
for $\kappa \leq 8$.
(Note that now walks are not constrained to reach the boundary of the ball of radius $n$, as would be the case
in the original Lawler-Schramm-Werner model.)

For higher dimensions (that is, $(\Z/n\Z)^d$) 
the sub-critical and super-critical behavior should be the same as the planar case;
\ie constant for sub-critical and order of the volume in the super-critical case.
The results of Brydges \& Spencer 
\cite{BS85} and Hara \& Slade \cite{HaraSlade1,HaraSlade2, HaraSlade3}
tell us that the the number of self-avoiding walks in $\Z^d , d \geq 5$, started at $0$, of length $n$, 
is asymptotic to $A \mu^n$, 
for some constant $A>0$, and that a typical self-avoiding walk of length $n$ behaves like 
a random walk.  This leads us to expect that on the finite torus $(\Z/n\Z)^d , d \geq 5$ the self-avoiding walk
should not ``feel'' its own presence until reaching length $n^{d/2}$ (this is just the birthday paradox).
Thus, we expect that for $d \geq 5$ taking $G_n = (\Z/n\Z)^d$, we have that
$$ L(x_c, G_n) = \Theta (n^{d/2}) = \sqrt{ |G_n| } . $$
Here $x_c$ is the critical parameter $x_c= \mu(\Z^d)^{-1}$.

The complete graph together with the Euclidean case give two possible notions of ``mean field behavior'':
critical exponent $\gamma = 1$ or expected length $L$ of order square-root of the volume.
We turn to the large girth case to understand which of these is the correct picture when there is more 
exotic geometry involved.

\subsection{$\SAW$ on large girth}

Expander graphs are graphs with very good expansion properties.
It is known that a random $d$-regular graph is a very good expander with high probability,
and has large girth around typical vertices with high probability.
In many models random $d$-regular graphs exhibit the same behavior as the mean-field complete graph
case. Sometimes such mean-field behavior can also be shown for expanders of large girth in general.
See the discussion immediately following Theorem \ref{thm: large girth}.

This experience may lead one to conjecture that the same phenomena will occur 
for the $\SAW$ model. 
The sub-critical and super-critical cases indeed do have constant and linear order expected length respectively.
However, the analogy breaks down 
for the expected length $L$ in the critical case, and the scaling is not $\sqrt{|G_n|}$.
That being said, a different measure of ``mean field'' is the critical exponent $\gamma$.
In this case we will show that $\gamma=1$ for graphs of large girth. 

Instrumental in the proofs are non-backtracking random walks.
A {\bf non-backtracking walk} is a path in a graph
that never backtracks the last edge it passed through.
A non-backtracking random walk is one that chooses each step randomly out of the currently
allowed steps.  This is a Markov chain on the set of directed edges of the graph.
For a non-backtracking random walk on a graph $G$ one may define the mixing time by
$$ \tau = \min \set{ t \ : \ \max_{u,v} | \Pr_u [\NB(t) = v ] - \tfrac1{|G|} | \leq \tfrac{1}{2|G|} } , $$
where $\NB$ is the non-backtracking random walk, and $\Pr_v$ is the associated probability measure 
conditioned on $\NB(0)=v$. In \cite{ABLS07} the mixing time of non-backtracking random walks
is studied, and it is shown that it is always better than the mixing time of the usual simple random walk.
We use the mixing time of the non-backtracking random walk to quantitatively define the notion of ``large girth''.

\begin{thm}
\label{thm: large girth}
Let $(G_n)_n$ be a sequence of $d$-regular transitive graphs with sizes $|G_n| \nearrow \infty$.
Let $g_n$ be the girth of $G_n$ and assume that $g_n \to \infty$ as $n \to \infty$.
Consider the $ \SAW$ model on the sequence $(G_n)_n$.

Then, $\tfrac1{d-1}$ is a critical (constant) sequence, and we have:
\begin{itemize}
\item If $x_n \to x$ for $x< (d-1)^{-1}$ then
$$ \lim_{n \to \infty} L(x_n,G_n) = \frac{(d-1) x}{1-(d-1)x} \AND \lim_{n \to \infty} I(x_n,G_n) = 
1- \tfrac{1}{d}  . $$

\item If $x = (d-1)^{-1}$ then,
$$ C^{-1} g_n \leq L(x,G_n) \leq C \cdot g_n (d-1)^{2 g_n}  $$
where $C>0$ is a constant that depends only on the degree $d$.

\item Assume that $G_n$ has large girth in the sense that if $\tau_n$
is the mixing time of the non-backtracking random walk on $G_n$,
then $\tau_n = o((d-1)^{g_n/4} )$ as $n \to \infty$.

If $x_n \to x > (d-1)^{-1}$ then there exists $c = c(x,d) > 0$ such that
$$ \liminf_{n \to \infty} \frac{ L(x_n,G_n) }{|G_n|} \geq c \AND \lim_{n \to \infty} I(x_n,G_n) = 0 . $$
\end{itemize}
\end{thm}

Note that one may find expander graphs $(G_n)_n$ with girth as large as $g_n = \eps \log |G_n|$ for 
any $\eps>0$. 
By results of \cite{ABLS07}, the mixing time of the non-backtracking random walk 
is faster than the mixing time of the simple random walk on $G_n$.
Choosing $(G_n)_n$ as these large-girth expander graphs, the mixing time in this case is $\tau_n = O(\log |G_n|)$.
So in Theorem \ref{thm: large girth} the critical expected length could be bounded by 
order $|G_n|^\delta$ for any choice of $\delta>0$ (and in fact may even be poly-logarithmic).
For example, it is well known that the above conditions hold for the 
constructions by Margulis \cite{Mar82} and the
Lubotzky-Phillips-Sarnak expanders \cite{LPS88}.

Occasionally one expects expanders, especially large girth expanders, 
to exhibit tree-like or mean-field behavior.
To illustrate this, let us contrast our results with some other mean-field behavior in the large girth setting.
In \cite{Nach09} finite graphs with large girth are shown to exhibit mean-field behavior of the critical window 
for bond percolation. Coincidently, the non-backtracking random walk appears as part of the conditions in that paper
as well. A more recent paper \cite{NP12} deals with the {\em infinitely many infinite components} phase in percolation
on infinite graphs with large girth, as well as the critical exponents for percolation and self-avoiding walks on such graphs.  
There it is shown that the exponents are mean-field exponents for large girth infinite graphs.
Again, there is a connection to the non-backtracking random walk, which appears in the proofs.

\subsection{Critical exponent in large girth}

As mentioned above,
although the expected length $L$ does not exhibit ``mean field behavior'' for large girth expanders, 
we may consider another measure of ``mean field behavior'', namely the ``critical exponent'' $\gamma$.
We show that $\gamma=1$ in the critical regime for graphs with large girth.  We also determine $\gamma$
for the sub- and super-critical regimes.

\begin{thm}
\label{thm: gamma}
For a sequence of $d$-regular transitive 
graphs $(G_n)_n$, of size $|G_n| \nearrow \infty$ and girth $g_n \to \infty$ as $n \to \infty$:

\begin{itemize}

\item
For the critical sequence $x_n = \tfrac{1}{d-1}$, the critical exponent $\gamma$ has the mean field value
$$ \lim_{n \to \infty} \gamma(x_n,G_n) =  1 . $$

\item
If $x_n \to x < \tfrac1{d-1}$ then 
$$ \lim_{n \to \infty} \gamma(x_n,G_n) = \frac{ \log \tfrac{d}{d-1} - \log (1-(d-1)x) }{
 - \log (1-(d-1)x) } \in (1,\infty) . $$

\item
If $x_n \to x > \tfrac1{d-1}$ and if the girth of $G_n$ is large enough so that 
$\tau_n = o((d-1)^{g_n/4})$ as $n \to \infty$ 
(where $\tau_n$ is the mining time of the non-backtracking random walk on $G_n$, 
as in the assumptions of %Lemma \ref{lem: super critical}), 
Theorem \ref{thm: large girth}),
then 
\begin{align*}
\lim_{n \to \infty} \gamma(x_n,G_n) & = \infty .
\end{align*}

\end{itemize}
\end{thm}

Recall that in the classical setting $I(x,G) \sim  c (1-x\mu)^{\gamma-1}$
(as in Section \ref{scn: critical exp}).  In our (large girth) setting the analogous quantity
for a sequence $x_n \to x < \tfrac{1}{d-1}$  is
$$ (1- (d-1) x_n)^{\gamma(x_n,G_n) -1} = 
\exp \big(  \log \tfrac{d}{d-1} \cdot \tfrac{ \log (1- (d-1) x_n) }{ -\log (1-(d-1)x ) } \big) 
\to \tfrac{d-1}{d} > 0 . $$
Also, by Theorem \ref{thm: large girth}, $I(x_n,G_n) \to \tfrac{d-1}{d}$.
This coincides nicely with the fact that $\gamma=1$ in the critical case.

\subsection{Some more general results}
\label{scn: general results}

Our analysis also provides some general results relating the trivial-intersection probability $I(x,o,G)$
to the expected length $L(x,o,G)$ and partition function $Z_{o,G}(x)$, for transitive graphs.
This relation also holds for infinite graphs (as long as the partition function converges).

\begin{prop}
\label{prop: intersection}
Suppose that $G$ is a transitive graph (finite or infinite).  
Let $x>0$ be such that $Z_G(x)$ converges (all $x>0$ in the finite graph case).
Then,
$$ L(x,G) + 1 = I(x,G) \cdot Z_G(x) . $$
\end{prop}

We also prove that the expected length $L$ is an increasing function of the parameter.
This result also holds for infinite graphs as long as the model is defined.

\begin{thm}
\label{thm: L increasing}
The function $x \mapsto L(x,o,G)$ is a non-decreasing function
(defined for all $x$ such that $Z_{o,G}(x) < \infty$).
\end{thm}

%
%In light of Theorem \ref{thm: large girth}, it is perhaps reasonable to augment the definition of the connective constant 
%for finite $d$-regular graphs of large girth.
%The following has been suggested by Itai Benjamini.
%\begin{ques}
%Let $(G_n)_n$ be a sequence of finite $d$ regular graphs with sizes $|G_n| \to \infty$ as $n \to \infty$.
%Suppose that $g_n$ is the girth of $G_n$, and $g_n \to \infty$.
%
%Find a sequence $(\eps_n)_n$ such that $\eps_n \to 0$ as $n \to \infty$ and such that for $x_n = \frac{1}{d-1} + \eps_n$,
%$$ C^{-1} \leq \liminf_{n \to \infty} \frac{ \E_{x_n,n} [ |\o| ] }{\sqrt{|G_n|} } \leq \limsup_{n \to \infty} \frac{ \E_{x_n,n} [ |\o| ] }{\sqrt{|G_n|} }  \leq C , $$
%where $C>0$ is some constant independent of $n$.
%\end{ques}
%

\section{Open questions}

Before moving to the proofs, let us mention some basic open questions.

\subsection{Critical exponents}

There are other critical exponents which appear in the classical self-avoiding walk literature
(\eg mean-square displacement exponent $\nu$).  It would be of interest to define these values
in the finite graph setting, and perhaps obtain some relations similar to the Fischer relations (see \eg \cite{hugo}).

Regarding the exponent $\gamma$, note that {\em a-priori} its value depends on the choice of the 
sequence of parameters $x_n$.

\begin{ques}
Show that for two critical sequences $\frac{x_n}{y_n} \to 1$, we have that
$$ \frac{\gamma(x_n,o_n,G_n) }{ \gamma(y_n,o_n,G_n) } \to 1 . $$
\end{ques}

Another phenomena arising from this work is that we expect $\gamma = \infty$ in the super-critical regime.
This regime has no analogue in the infinite-graph case.

\begin{ques}
Show that if $(x_n)_n$ is a super-critical sequence then $\gamma(x_n,o_n,G_n) \to \infty$.
\end{ques}

\subsection{$I$ is decreasing}

The expected length of a self-avoiding walk $L(x,o,G)$ is an increasing function of $x$.
Intuitively, raising $x$ increases the length of a typical path.  
We would expect that this phenomena should imply that two independent self-avoiding walks 
should have smaller chances of a trivial intersection as $x$ grows.

\begin{ques}
Show that the function $x \mapsto I(x,o,G)$ is a non-increasing function of $x$.
\end{ques}

\subsection{Extending to other models}

In this paper we have dealt with transitive large girth graphs.
Transitivity is required to relate the expected length $L$ to the intersection probability $I$, 
as in Proposition \ref{prop: intersection}, which is a starting point for our analysis.
The large girth is required for the connection to the non-backtracking walk, see Section \ref{scn: connection to NB} below.

It is very natural to extend the results in this paper and investigate graphs that are not necessarily transitive, 
and graphs that may have small cycles, but only very few.  For example, this is what random $d$-regular graphs look like.
It seems that the analysis becomes more difficult in these cases, and perhaps new ideas are required.

\section{Connection to non-backtracking random walk}
\label{scn: connection to NB}

The connection between self-avoiding walks and non-backtracking walks has 
already appeared in a paper by Kesten \cite{Kesten}.
It is used to give bounds on the connective constant of $\Z^d$ for large $d$.

It seems, in fact, that there is a deeper connection between self-avoiding walks and non-backtracking random walks.
The theory is developed in \cite{Fitzner}, where Fitzner and van der Hofstad develop lace expansion based on the 
non-backtracking random walk rather than the simple random walk.  They apply this in $\Z^d$ to the analysis of 
$\SAW$.  However, it seems that lace expansion methods, while achieving amazing results in $\Z^d$, are difficult to
generalize to other settings (such as non-commutative groups and even more general graphs).

We now give a simple application of  
non-backtracking random walks to the study of self avoiding walks.
%In a graph of large girth, since the local environment is tree-like, 
%the non-backtracking random walk has to start out as a self-avoiding walk
%and only after going far away may it return to a vertex already visited.
%This leads to a simple connection between self-avoiding walks and non-backtracking random walks.

A non-backtracking path $\o$ in $G$ is a sequence $\o = (\o_0,\ldots,\o_n)$ such that for all 
$j<n$ we have $\o_j \sim \o_{j+1}$ (so that $\o$ is a path in $G$), and this path never backtracks an edge in $G$:
\ie for all $j<n-1$ we have that $\o_{j+2} \neq \o_j$.
We use $\NB$ to denote a non-backtracking random walk on a fixed 
$d$-regular graph $G$ started at a fixed vertex $o \in G$. 
$\bP_o,\bE_o$ denote the corresponding probability measure and expectation.
%$\bP$, without subscripts refers to starting from a uniformly chosen vertex; $\bP = \tfrac1n \sum_v \bP_v$.
%We assume that $|G|=n$.
Thus, for any non-backtracking path $\o$ in $G$ of length $|\o|=k$ started at $o$,
the $\bP_o$-probability that $\NB[0,k] = \o$ is $\tfrac1{d} \cdot (d-1)^{-(k-1)}$. 
It is simple to see that this is obtained by the Markov chain on directed edges, 
that at each step chooses uniformly among all
edges emanating from the current vertex, except the edge that was used to reach the current vertex in the 
previous time step.

The key observation is as follows.
$$ |\SAW_k(o,G)| \cdot (d(d-1)^{k-1} )^{-1} = \sum_{\o \in \SAW_k} \tfrac{1}{d (d-1)^{k-1} } 
= \bP_o [ \NB[0,k] \in \SAW(o,G) ] . $$
Define 
$$ T = \inf \set{ k \geq 0 \ : \ \NB[0,k] \not\in \SAW(G) } = \inf \set{ k > 2 \ : \ \exists \ j < k \ , \ \NB(j) = \NB(k) } . $$
We call $T$ the self-intersection time of $\NB$. 
Thus,
\begin{align}
\label{eqn: key}
|\SAW_k(o,G)| & = \frac{d}{d-1} \cdot (d-1)^k \cdot \bP_o [ T > k ]
\end{align}

Now consider the partition function for any $0<y \neq 1$ and $x = \tfrac{y}{d-1}$,
\begin{align} \label{eqn: partition function}
Z_{o,G}(x) & = \sum_{k=0}^{\infty} |\SAW_k(o,G)| (d-1)^{-k} y^k \nonumber  \\
& = 
\tfrac{d}{d-1} \cdot \sum_{k=0}^{\infty} \bP_o [ T > k ] y^k  = \frac{d}{d-1}  \cdot \frac{1}{y-1} \cdot \bE_o [ y^T - 1 ] .
\end{align}
%So,
%\begin{align} \label{eqn: mu}
%\Pr_{x,o,G} [ |\o| < m ] & = (Z_{o,G}(y))^{-1} \cdot \frac{d}{d-1} \cdot \E_o \sum_{k=0}^{(T \mn m) -1} x^k
% = \frac{ \E_o [ x^{T \mn m} - 1 ] }{ \E_o [ x^T - 1 ] } . 
%\end{align}
%Moreover, 
%we may also use the fact that
Since,
\begin{align*}
\sum_{y=0}^{m-1} k y^k & = y \frac{\p}{\p y} \sum_{k=0}^{m-1} y^k = y \frac{\p}{\p y} \frac{y^m-1}{y-1} \\
& = y \frac{my^{m-1} (y-1) - y^m+1 }{(y-1)^2} 
= \frac{y}{y-1} \cdot m y^{m-1} - \frac{y}{y-1} \cdot \frac{y^m-1}{y-1} ,
\end{align*}
we get that
\begin{align} \label{eqn: length}
\E_{x,o,G} [|\o|] & = (Z_{o,G}(x))^{-1} \cdot \tfrac{d}{d-1} \cdot \bE_o [ \sum_{k=0}^{\infty} k y^k \1{T>k} ] 
\nonumber \\
& =  \frac{\bE_o [T y^{T} ] }{\bE_o[y^T-1] } - \frac{y}{y-1} = \frac{y}{1-y} - \frac{ \bE_o[Ty^T]}{ \bE_o [ 1-y^T] } . 
\end{align}

For the case $x = \tfrac{1}{d-1}$ (\ie $y=1$),
\begin{align} \label{eqn: partition function at 1}
Z_{o,G}(x) & = \sum_{k=0}^{n-1} |\SAW_k(o,G)| (d-1)^{-k} \nonumber   \\
& = 
\tfrac{d}{d-1} \cdot \sum_{k=0}^{n-1} \bP_o [ T > k ]   = \frac{d}{d-1} \cdot \bE_o [ T ] .
\end{align}
And similarly to the above,
%\begin{align} \label{eqn: mu at 1}
%\Pr_{y,o,G} [ |\o| < m ] & = (Z_{o,G}(y))^{-1} \cdot \frac{d}{d-1} \cdot \E_o \sum_{k=0}^{(T \mn m) -1} 
% = \frac{ \E_o [ (T \mn m)  ] }{ \E_o [ T  ] } , 
%\end{align}
\begin{align} \label{eqn: length at 1}
\E_{x,o,G} [|\o|] & = (Z_{o,G}(x))^{-1} \cdot \tfrac{d}{d-1} \cdot \bE_o [ \sum_{k=0}^{\infty} k \1{T>k} ] 
 =  \frac{\bE_o [T(T-1) ] }{2 \bE_o[T] } .
\end{align}

\section{Proofs of General Results}

We start with the proof of Proposition \ref{prop:unique}, which is the statement the critical 
sequences are unique (in a certain sense).

\begin{proof}[Proof of Proposition \ref{prop:unique}] 
To simplify the notation in this proof, we write $Z(x_n)$ for $Z_{o_n,G_n}(x_n)$.

For the first direction, assume that both $(x_n)_n, (y_n)_n$ are critical sequences.
It suffices to prove that $\limsup_{n \to \infty} \tfrac{x_n}{y_n} \leq 1$, since both sequences play the same role.

If $\limsup_{n \to \infty} \tfrac{x_n}{y_n} > 1$, then there exists some $0<\eps < \tfrac14$ and subsequences
$(x_{n_k})_k , (y_{n_k})_k$ such that for all $k$,
$$ \frac{x_{n_k}}{y_{n_k}} \geq 1+4 \eps \geq 1 + \tfrac{2\eps}{1-\eps} = \frac{1+\eps}{1-\eps} . $$
Thus, by the definition of a critical sequence,
using the fact that $Z_{o,G}(\cdot)$ is an increasing function,
\begin{align*}
\infty & > \limsup_{n \to \infty} Z(x_n(1-\eps) ) \geq \limsup_{k \to \infty} Z(x_{n_k}(1-\eps) ) \\
& \geq \limsup_{k \to \infty} Z(y_{n_k}(1+\eps)) \geq 
\liminf_{n \to \infty} Z(y_n(1+\eps)) = \infty ,  
\end{align*}
a contradiction!

For the other direction, let $(x_n)_n$ be a critical sequence and $(y_n)_n$ a sequence such that $\frac{x_n}{y_n} \to 1$.

Fix small $\eps>0$.  Then, for all large enough $n$ we have $x_n(1-\tfrac12 \eps) \leq y_n \leq (1+\eps) x_n$.
Thus, 
\begin{align*}
\limsup_{n \to \infty} Z(y_n(1-\eps)) & \leq \limsup_{n \to \infty} Z(x_n(1-\eps^2) ) < \infty , \\
\liminf_{n \to \infty} Z(y_n(1+\eps) ) & \geq \liminf_{n \to \infty} Z(x_n(1 + \tfrac12 \eps(1-\eps) ) ) = \infty .
\end{align*}
So $(y_n)_n$ is critical as well.
\end{proof}

We move to the proof of Proposition \ref{prop: intersection}, which is the identity $I Z = L+1$.

\begin{proof}[Proof of Proposition \ref{prop: intersection}]
Let 
$$ S_{k,j} : = \set{ (\o,\o') \in \SAW_k(o,G) \times \SAW_j(o,G) \ : \ \o \cap \o' = \set{ o} } . $$
For $(\o,\o') \in S_{k,j}$ let $\hat{\o}$ be the path that is $\vphi(\o \cup \o')$ mapped by the automorphism $\vphi$
of $G$ taking $\o(k)$ to $o$.  Note that $\hat{\o} \in \SAW_{k+j}(o,G)$ because $\o \cap \o' = \set{o}$.
This map may be inverted.  If $\hat{\o} \in \SAW_{k+j}(o,G)$ one can define a pair 
$$ \o = \vphi^{-1} (\hat\o (k) , \hat\o(k-1) , \ldots, \hat{\o} (0) ) \AND \o' = \vphi^{-1} (\hat{\o}[k,k+j]) , $$  
where $\vphi$ is the same automorphism as before.  Then $(\o,\o') \in S_{k,j}$.
Thus, we have shown that when $G$ is transitive, $|S_{k,j}| = |\SAW_{k+j}|$.

We now compute for $x = \tfrac{y}{d-1}, y \neq 1$,
\begin{align*}
(Z_{o,G}(x))^{2} & \cdot \Pr [ \o \cap \o' = \set{o} ] = \tfrac{d}{d-1} \cdot \sum_{k,j=0}^{n-1} y^{k+j} \bP_o[ T>k+j] = 
\tfrac{d}{d-1} \cdot \bE_o \sum_{k=0}^{T-1} y^k \sum_{j=0}^{T-1-k} y^j \\
& = \tfrac{d}{d-1} \cdot \frac{1}{y-1} \cdot \bE_o \sum_{k=0}^{T-1} y^k(y^{T-k}-1) = \tfrac{d}{d-1} \cdot \frac{1}{(y-1)^2} \cdot
\bE_o [ T y^T (y-1) - y^T + 1 ] . 
\end{align*}
Thus, 
\begin{align*}
\Pr [ \o \cap \o' = \set{o}] & = \frac{d-1}{d} \cdot \frac{\bE_o [ T y^T (y-1) - y^T + 1 ]}{(\bE_o [y^T - 1])^2 } .
\end{align*}
So by \eqref{eqn: length},
\begin{align*}
\E_{x,o,G} [ |\o| ] +1& = \frac{\bE_o [T y^{T} ] }{\bE_o[y^T-1] } - \frac{1}{y-1} = 
\frac{1}{y-1} \cdot \frac{\bE_o [T y^{T} (y-1) - y^T + 1 ] }{\bE_o[y^T-1] } \\
& = \frac{1}{y-1} \cdot \frac{d}{d-1} \cdot \Pr [ \o \cap \o' = \set{o} ] \cdot \bE_o [ y^T-1 ] 
\\
& = \Pr [ \o \cap \o' = \set{o} ] \cdot Z_{o,G}(x) . 
\end{align*}
In the case $x = \tfrac{1}{d-1}$ ($y=1$),
\begin{align*}
(Z_{o,G}(x))^{2} \cdot \Pr [ \o \cap \o' = \set{o} ] & =  \tfrac{d}{d-1} \cdot \sum_{k,j=0}^{n-1} \bP_o[ T>k+j] 
= \tfrac{d}{d-1} \cdot \bE_o \sum_{k=0}^{T-1} (T-k) \\
& = \tfrac{d}{d-1} \cdot \bE_o \sum_{k=1}^{T} k = \tfrac{d}{d-1} \cdot \tfrac12  \bE_o [ T(T+1) ] .
\end{align*}
So,
\begin{align*}
\Pr [ \o \cap \o' = \set{o} ] & = \frac{d-1}{d} \cdot \frac{ \bE_o [ T(T+1) ] }{ 2 (\bE_o [ T ] )^2 } , 
\end{align*}
and by \eqref{eqn: length at 1},
\begin{align*}
\E_{x,o,G} [|\o|] +1& = \frac{ \bE_o [ T(T+1)] }{ 2 \bE_o [ T ] } = \frac{d}{d-1}  \cdot \Pr [ \o \cap \o' = \emptyset ]
\cdot \bE_o [ T ] \\
& = \Pr [ \o \cap \o' = \set{o} ] \cdot Z_{o,G}(x) . 
\end{align*}
\end{proof}

We now prove Theorem \ref{thm: L increasing}, showing that $L(x)$ is an increasing function of $x$,
regardless of the graph chosen (as long as it is defined).

\begin{proof}[Proof of Theorem \ref{thm: L increasing}]
%We compute that derivative of $L$.
Define $P_m(x) = \Pr_{x,o,G} [ |\o| < m ]$.
We claim that $P_m$ is decreasing in $x$.
This implies that for $x<z$, the random variable
$|\o|$ under $\Pr_{x,o,G}$ is stochastically dominated by $|\o|$ under $\Pr_{z,o,G}$.
So $L(x) = \E_{x,o,G}[|\o|] \leq \E_{z,o,G}[|\o|] = L(z)$ for all $x<z$.

We move to prove that $P_m$ is a decreasing function.
Write $y=(d-1)x$ and $Q_m(x) = \tfrac{d}{d-1} \bE_o \sum_{k=0}^{(T\mn m)-1} y^k$.
Using \eqref{eqn: key}, we have that $P_m(x) = Q_m(x) / Z(x)$.
So $P_m'(x) = Z(x)^{-2} \cdot \sr{ Q_m'(x)  Z(x) - Q_m(x) Z'(x) }$.

Note that $Z = Q_{\infty}$.
We have that for any $m \in \N \cup \set{ \infty}$,
\begin{align*}
y Q_m'(x) & = \tfrac{d}{d-1} \cdot \bE_o \sum_{k=0}^{(T \mn m)-1} (d-1) k y^{k} 
= (d-1) \E_{x,o,G} [ |\o| \1{|\o| < m } ] \cdot Z(x) . 
\end{align*}
Thus, $P'_m(x) \leq 0$ if and only if
$$ \E_{x,o,G} [ |\o| \1{|\o| < m }] \leq \E_{x,o,G} [|\o|] \cdot P_m(x)
= \E_{x,o,G} [|\o| \1{|\o|<m } ] \cdot P_m(x) + \E_{x,o,G} [|\o| \1{ |\o| \geq m} ] \cdot P_m(x) , $$
which is if and only if
$$ \E_{x,o,G} [ |\o| \1{ |\o| < m } ] \cdot (1-P_m(x)) \leq \E_{x,o,G} [ |\o| \1{ |\o| \geq m }] \cdot P_m(x) . $$
Recall that $P_m(x) = \Pr_{x,o,G} [ |\o| < m ]$. If $P_m(x) \in \set{0,1}$ then the inequality above holds trivially.
If $0 < P_m(x) < 1$ then the inequality is equivalent to
$$ \E_{x,o,G} [ |\o| \ | \ |\o| < m ] \leq \E_{x,o,G} [ |\o| \ | \ |\o| \geq m ] , $$
which is obviously true always.
\end{proof}

Finally we prove Theorem \ref{thm: critical and I}, which states that sequences are super-critical
(resp.\ sub-critical) if and only if $I \to 0$ (resp.\ $I >0$).

\begin{proof}[Proof of Theorem \ref{thm: critical and I}]
Jensen's inequality tells us that for any $x>0$, $y=(d-1)x$ and $\alpha \neq - \log y$,
\begin{align*}
e^{ \alpha L(x,G) } & \leq \E_{x,G} [ e^{\alpha |\o| } ] = 
Z_G(x)^{-1} \cdot \tfrac{d}{d-1} \cdot \bE_o \sum_{k=0}^{T-1} e^{\alpha k} y^k 
= \frac{ Z_G(x e^{\alpha} ) }{ Z_G(x) } . 
\end{align*}

Now, if $(x_n)_n$ is a sub-sequence such that $Z_{G_n}(x_n) \to \infty$,
then choosing $\eps>0$ and $z_n = x_n(1-\eps)$, we have with $\alpha = \log(1-\eps)$,
\begin{align*}
I(x_n,G_n) & = \frac{ L(x_n,G_n) + 1 }{ Z_{G_n}(x_n) } 
\leq \frac{ -\log(1-\eps) + \log Z_{G_n}(x_n) - \log Z_{G_n} (z_n) }{ -\log(1-\eps) \cdot Z_{G_n}(x_n) } \to 0 . 
\end{align*}
Thus, $Z_{G_n}(x_n) \to \infty$ implies that $I(x_n,G_n) \to 0$.

Also, since
$$ I(x_n,G_n) = \frac{ L(x_n,G_n) +1}{Z_{G_n}(x_n) } \geq \frac{1}{Z_{G_n}(x_n) } , $$
we have that $I(x_n,G_n) \to 0$ implies that $Z_{G_n}(x_n) \to \infty$.

Thus, 
\begin{align*}
\liminf_{n \to \infty} Z_{G_n}(x_n) = \infty & \iff \limsup_{n \to\infty} I(x_n,G_n) = 0 , \\
\limsup_{n \to \infty} Z_{G_n}(x_n) < \infty & \iff \liminf_{n \to \infty} I(x_n,G_n) > 0 . 
\end{align*}
\end{proof}

\section{Mean-field $\SAW$}
\label{scn: mean field}

In this section we give a short account of the mean-field $ \SAW$; that is when $G_n = K_n$ is the complete 
graph on $n$ vertices.  
%In particular, we prove Theorem \ref{thm: mean field}.
%The proof of Theorem \ref{thm: mean field} is an immediate combination of Lemmas 
%\ref{lem: super-critical MF}, \ref{lem: critical MF} and \ref{lem: sub-critical MF} below.

The main (simple) observation is that on the complete graph 
one can count exactly the number of self-avoiding walks of a specific length.
Specifically, for a fixed vertex $o_n \in K_n$,
%for every $k$ let $\Gamma_{k,n}$ be the set of all self avoiding walks of length $n$ in
%the complete graph on $n$ vertices, $G_n$.
%Then,
$$ %|\Gamma_{k,n} | = 
\SAW_k(o_n,K_n) =
\begin{cases}
\frac{(n-1)!}{(n-1-k)!} & 0 \leq k \leq n-1 \\
0 & k > n-1 .
\end{cases}
$$

Thus, for all $0 \leq k \leq n-1$,
\begin{align*}
\Pr_{x,G_n} [ |\o| = n-1- k] & = \frac{ x^{n-1-k} |\SAW_{n-1-k}(o_n,K_n) | }{ \sum_{j=0}^{n-1} x^j |\SAW_j(o_n,K_n) | }
= \frac{ x^{n-1-k} |\SAW_{n-1-k}(o_n,K_n) | }{ \sum_{j=0}^{n-1} x^{n-1-j} |\SAW_{n-1-j}(o_n,K_n) | }
\\
& = \frac{ \frac{x^{-k}}{k!} }{ \sum_{j=0}^{n-1} \frac{x^{-j} }{j! } }  = \Pr [ P = k \ | \ P \leq n-1 ] ,
\end{align*}
where $P \sim \mathrm{Poi}(x^{-1} )$.
That is, under $\Pr_{x,K_n}$ the quantity $n-1-|\o|$ has a conditional Poisson distribution.
Specifically,
\begin{align}
\label{eqn: Poisson MF}
L(x,K_n) & = n-1 - \E [P \ | \ P \leq n-1 ] .
\end{align}
This leaves the task of understanding the conditional expectation above for the different regimes of $x$.

Since this is the main observation required to prove Theorems \ref{thm: mean field} and \ref{thm: gamma MF}, 
we have placed the proofs of these theorems in the appendix.

\section{$\SAW$ on large girth graphs}

We now investigate the consequences of the connections between $\SAW$ and non-backtracking random walk
made in Section \ref{scn: connection to NB}, for sequences of graphs with large girth.

For this section let $(G_n)_n$ be a sequence of $d$-regular graphs of size $|G_n| = n$ and 
girth of $G_n$ being $g_n$.  Assume that $g_n \to \infty$ as $n \to \infty$.
We fix some vertex $o_n \in G_n$, and denote 
%$\mu_{y,n} = \mu_{y,o_n,G_n}, \bE_{y,n} = \bE_{y,o_n,G_n}$.
$L_n(x) = L(x,o_n,G_n) , I_n(x) = I(x,o_n,G_n) , Z_n(x) = Z_{o_n,G_n}(x)$.
Sometimes we will omit the subscript $n$ when it is clear from the context.

\subsection{Sub-critical phase}

\begin{lem}[Sub-critical phase]
\label{lem: sub}
Let $(x_n)_n$ be a positive sequence such that $\lim_{n \to \infty} x_n = x < \tfrac1{d-1}$.
Then, %for $y_n = \tfrac{x_n}{d-1}$,
$$ \lim_{n \to \infty} L_n(x_n) = \frac{(d-1)x}{1-(d-1)x} . $$
\end{lem}

\begin{proof}
We will prove that in fact that if $x_n < \tfrac{1}{d-1}$, then for $y_n = (d-1)x_n$,
\begin{align}
\label{eqn: sub critical length}
\frac{y_n}{1-y_n} - \frac{g_n y_n^{g_n} }{1- y_n^{g_n} } & \leq L_n(x_n) \leq \frac{y_n}{1-y_n} . 
\end{align}
Indeed, using \eqref{eqn: length}, we only need to bound $$\frac{\E_{o_n}[Ty^T] }{ \E_{o_n}[1-y^T] }$$ 
from above for $y<1$,
where $T$ is the self intersection time of a non-backtracking random walk on $G_n$.
Note that $T \geq g_n$ a.s.  Also, $k \mapsto k y^k$ is decreasing for $k>(-\log y)^{-1}$.
So $\E_o[Ty^T] \leq g_n y^{g_n}$ when $g_n$ is large enough, and $\E_o [ 1-y^T ] \geq 1-y^{g_n}$.

The lemma now follows because $y_n < 1$ for large enough $n$ and $g_n \to \infty$.
\end{proof}

\subsection{Super-critical phase}

\begin{lem}
\label{lem: lower bound}
Suppose that $G$ is a transitive graph with girth $g$ and size $|G|=n$.
Let 
$$ \tau = \min \set{ t \ : \ \max_{u,v} | \bP_u [\NB(t) = v ] - \tfrac1n | \leq \tfrac{1}{2n} } $$
be the mixing time of the non-backtracking random walk $\NB$ on $G$,
and let $T$ be the self intersection time of $\NB$.
For all $m \geq \tau$,
$$ \bP_o [ T>k+m ] \geq \sr{ 1 - \tfrac{3(k+1) m}{2n} - m^2 \cdot (d-1)^{- \lfloor g_n / 2 \rfloor } } \cdot \bP_o [ T>k] . $$
\end{lem}

\begin{proof}
%We use $S(t) = \NB(t)$ for simplicity of the presentation.
We will make use of the fact that uniformly over $v$,
$$ \bP_v [ \NB(t) = v ] \leq
\begin{cases}
0 & \textrm{ if } t<g_n \\
(d-1)^{- \lfloor g_n/2 \rfloor } & \textrm{ if } g_n \leq t < \tau \\
\tfrac{3}{2n}  & \textrm{ if } t \geq \tau .
\end{cases}
$$
The middle case coming from the fact that $\set{ u \ : \ \dist(v,u) \leq \tfrac{g_n}{2} }$ is a tree,
so to hit $v$ at time $t$ with a non-backtracking walk, 
one must be at distance $\lfloor \tfrac{g_n}{2} \rfloor$ at time $t- \lfloor \tfrac{g_n}{2} \rfloor$
and then move towards $v$ in the next $\lfloor \tfrac{g_n}{2} \rfloor$ steps.

Fix $m \geq \tau$.
$\NB[0,k+m]$ is uniformly distributed on non-backtracking paths of length $k+m$ starting at $\NB(0)$.
Thus, conditioned on $\NB[0,k] , \NB(k+m) = v$, 
the walk $\NB[k,k+m]$ is uniformly distributed on non-backtracking paths of length $m$ starting at $\NB(k)$, 
such that the first step is not $\NB(k-1)$ and such that the last step is at $v$.
Fix $v$ and conditioned on $\NB[0,k]=A$. Set
$$
\O = \O(v,A) = \left\{  \o \textrm{ non-backtracking } \ : \
\begin{array}{l}
|\o|=m \ , \  \o(0)=\NB(k) \  , \  \o(1) \neq \NB(k-1) , \\ 
\o[1,m] \cap A \neq \emptyset \ , \ \o(m) = v
\end{array}
\right\}
$$
$$
\O' = \O'(v,A) = \left\{  \o \textrm{ non-backtracking } \ : \
\begin{array}{l}
|\o|=m \ , \  \o(0)=v  , \\  
\o[0,m-1] \cap A \neq \emptyset  
\end{array}
\right\}
$$
Since for every $\o \in \O$ we have that the reversal of $\o$ is in $\O'$ we get that $|\O| \leq |\O'|$.
Thus, using the fact that $m \geq \tau$,
\begin{align*}
\bP_o \big[ \NB[k+1,k+m] & \cap A \neq \emptyset 
 \ , \  \NB[0,k] = A , \NB(k+m) = v \big] \\ 
& = |\O| \cdot (d-1)^{-m} \cdot \bP_o [ \NB[0,k] = A , \NB(k+m) = v ] \\
& \leq |\O'| \cdot (d-1)^{-m} \cdot \bP_o [ \NB[0,k] = A , \NB(k+m) = v ] \\
& = \bP_v [ \NB[0,m-1] \cap A \neq \emptyset ] \cdot \bP_o[ \NB[0,k] = A , \NB(k+m) = v ] \\
& \leq \bP_v [ \NB[0,m-1] \cap A \neq \emptyset ] \cdot \bP_o [ \NB[0,k] = A  ] \cdot \frac{3}{2n} .
\end{align*}
Summing over $v$ we obtain
\begin{align*}
\bP_o \big[ \NB[k+1,k+m] \cap A \neq \emptyset & \  \big| \  \NB[0,k]=A \big] 
\leq \tfrac32 \cdot \bP_o [ \NB[0,m-1] \cap A \neq \emptyset ] .
\end{align*}
Since the uniform distribution is stationary for the non-backtracking random walk on a regular graph,
\begin{align*}
\tfrac{1}{n} \sum_{o \in G}  \bP_o [ \NB[0,m-1] \cap A \neq \emptyset ] & \leq \tfrac{1}{n} \sum_{o \in G} 
\bE_o\big[ | \NB[0,m-1] \cap A | \big]  
=  \tfrac{1}{n} \sum_{t=0}^{m-1}  \sum_{o \in G} \bP_o [ \NB(t) \in A ] = \frac{|A| m}{n} . 
\end{align*}
Thus,
$$ \tfrac{1}{n} \sum_{o \in G} 
\bP_o \big[ \NB[k+1,k+m] \cap \NB[0,k] \neq \emptyset  \  \big| \  \NB[0,k]  \big] \leq \frac{3 (k+1) m}{2n} . $$

Also, if $\NB[k,k+m] \not\in \SAW$ then there exist $0 \leq t < t' \leq m$ such that 
$\NB(k+t) = \NB(k+t')$.  We have seen above that for any 
such pair $t<t'$, this probability is bounded by
$\bP_o [ \NB(k+t) = \NB(k+t') \ | \ \NB[0,k]  ] \leq (d-1)^{- \lfloor g_n / 2 \rfloor}$.  Thus,
\begin{align*}
\bP_o \big[ \NB[k,k+m] \not\in \SAW \ \big| \ \NB[0,k] \big] 
& \leq \sum_{0 \leq t< t' \leq m} \Pr \big[ \NB(k+t) = \NB(k+t') \ \big| \ \NB[0,k] \big] \\
& \leq m^2 \cdot (d-1)^{-\lfloor g_n / 2 \rfloor} .
\end{align*}

If $\NB[k,k+m] \in \SAW , \NB[0,k] \in \SAW$ and $\NB[k+1,k+m] \cap \NB[0,k] = \emptyset$, 
then $\NB[0,k+m] \in \SAW$.
Thus, combining all the above we obtain that
$$ \tfrac{1}{n} \sum_{o \in G} \bP_o [ \NB[0,k+m] \in \SAW ] \geq 
\sr{ 1 - \tfrac{3(k+1) m}{2n} - m^2 \cdot (d-1)^{- \lfloor g_n / 2 \rfloor } } \cdot 
\tfrac{1}{n} \sum_{o \in G} \bP_o [ \NB[0,k] \in \SAW ] . $$
Since $G$ is transitive, $\Pr_o [ \NB[0,k] \in \SAW ]$ does not depend on the choice of starting vertex $o \in G$,
and we obtain the lemma.
\end{proof}

\begin{cor}
\label{cor: pk lower bound}
Let $(G_n)_n$ be a sequence of transitive graphs such that $|G_n| \nearrow \infty$.
Suppose that the mixing time of the non-backtracking random walk on $G_n$ satisfies 
$\tau_n = o( (d-1)^{g_n / 4} )$ as $n \to \infty$, where $g_n$ is the girth of $G_n$. 

Then, there exists a constant $c>0$ such that for every $\delta>0$
and all $k \leq \frac{\delta}{6} |G_n|$,
$\bP_o [ T>k] \geq c e^{- \delta k }$.
\end{cor}

\begin{proof}
Set $p_k = \bP_o[T>k]$.
If we choose $m=m_n=\tau_n$, we may take $n$ large enough so that 
$m^2 \cdot (d-1)^{-\lfloor g_n / 2 \rfloor } < \tfrac14$.
By Lemma \ref{lem: lower bound} (using the inequality $e^{-2 \xi} \leq 1 - \xi$ valid for all $\xi \leq \tfrac14$),
\begin{align*}
p_k & \geq \exp \sr{ - 2 m^2 \cdot (d-1)^{- \lfloor g_n/2 \rfloor } - \tfrac{3 (k-m) m}{|G_n|} - \tfrac{3m}{|G_n|} } \cdot p_{k-m} \\
& \geq \cdots \geq \exp \sr{ - \lfloor \tfrac{k}{m} \rfloor \cdot 2 m^2 \cdot (d-1)^{- \lfloor g_n/2 \rfloor } 
- \tfrac{3m}{|G_n|} \cdot \sum_{j=1}^{\lfloor k/m \rfloor } j m - \tfrac{3m}{|G_n|} \cdot \lfloor \tfrac{k}{m} \rfloor } \cdot p_m \\
& \geq \exp \sr{ - 2 k m \cdot (d-1)^{- \lfloor g_n / 2 \rfloor } - \frac{3 k^2}{|G_n|} - \frac{3k}{|G_n|} } \cdot 
\exp \sr{ - 2 m^2 \cdot (d-1)^{- \lfloor g_n/2 \rfloor } } .
\end{align*}
Thus, if $k \leq \beta |G_n|$ then,
\begin{align} \label{eqn: pk lower bound}
p_k & \geq e^{-1/2} \cdot \exp \sr{ - \tfrac{1}{2 m} k - 3 \beta k - \tfrac{3}{|G_n|} k } . 
\end{align}
Since $m = \tau_n \to \infty$ as $n \to \infty$, we have that if $\beta = \tfrac{\delta}{6}$ then for all $n$ large enough
(so that $2m >> \delta^{-1}$),
$p_k \geq e^{-1/2} \cdot e^{- \delta k}$.
\end{proof}

\begin{lem}[Super-critical phase] \label{lem: super critical}
Let $(G_n)_n$ be a sequence of transitive graphs such that $|G_n| \nearrow \infty$ with girth $g_n$.
Suppose that the mixing time of the non-backtracking random walk on $G_n$ satisfies 
$\tau_n = o( (d-1)^{g_n / 4} )$ as $n \to \infty$, where $g_n$ is the girth of $G_n$. 

Then, there exists a constant $c>0$ such that
for any $x > \tfrac1{d-1}$ there exists $n_0>0$ such that for all $n>n_0$,
$$ L_n(x) \geq c (1 \mn \log ((d-1)x) ) \cdot |G_n| . $$
\end{lem}

\begin{proof}
Set $y=(d-1)x > 1$.  Choose $\delta>0$ small enough so that $\delta = 1 \mn \tfrac12 \log y$.
For any integers $0 < m < M \leq \tfrac{\delta}{6} |G_n|$ we have by Corollary \ref{cor: pk lower bound}, 
\begin{align*}
\sum_{k=0}^m y^k \bP_o[T>k] & \leq \frac{y^{m+1} - 1}{y-1} , \\
\sum_{k=0}^{M} y^k \bP_o[T>k] & \geq c \cdot \sum_{k=0}^M \sqrt{y}^k = c \cdot \frac{\sqrt{y}^{M+1} -1}{\sqrt{y}-1} .
\end{align*}
Choose $M = \lfloor \tfrac{\delta}{6} |G_n| \rfloor$ and $m = \lfloor \tfrac{M-1}{4} \rfloor$,
so that $y^{m+1} \sqrt{y}^{-M-1} \leq y^{-m}$.
Note that for some universal constant $c>0$ we have that $m \geq c \delta |G_n|$.
We can compute:
\begin{align*}
\Pr_{x,o_n,G_n} [ |\o| \leq m ] & = \frac{\sum_{k=0}^m y^k \bP_o[T>k] }{ \sum_{k=0}^{|G_n|-1} y^k \bP_o[T>k] } 
\leq c^{-1} \cdot \frac{y^{m+1} - 1}{\sqrt{y}^{M+1} -1} \cdot \frac{\sqrt{y}-1}{y-1} \\
& \leq c^{-1} \cdot \frac{y^{m+1} }{\sqrt{y}^{M+1} } \cdot \frac{1}{1- \sqrt{y}^{-M-1} } 
\leq C \cdot y^{- c \delta |G_n|} , 
\end{align*}
for some constant $C>0$.
Thus,
\begin{align*}
L_n(x) & \geq (1 - C y^{-c \delta |G_n|} ) \cdot c \delta |G_n| . 
\end{align*}
\end{proof}

\subsection{Critical phase}

\begin{lem}[Critical phase] \label{lem: critical}
For a sequence of $d$-regular transitive 
graphs $(G_n)_n$, of size $|G_n|=n$ and girth $g_n \to \infty$ as $n \to \infty$,
we have that at the critical sequence $x_n = \tfrac{1}{d-1}$,
$$ \tfrac12 \bE_o[T-1] \leq L_n(x_n) \leq \tfrac{d}{d-1} \bE_o [T] - 1 .  $$
Specifically, $L_n(x_n) \to \infty$ and
$L_n(x_n) \leq \tfrac{d}{d-1} (1+o(1) )  g_n (d-1)^{g_n} . $ 
So $L_n$ has neither sub-critical nor super-critical behavior.
\end{lem}

\begin{proof}
Using \eqref{eqn: partition function at 1} and \eqref{eqn: length at 1} we have that
$Z_{n}(\tfrac{1}{d-1}) = \tfrac{d}{d-1} \bE_o [T] \to \infty , $
and 
$$ L_n(\tfrac1{d-1}) = \frac{\bE_o[T^2] - \bE_o[T] }{2 \bE_o[T] } 
\geq \tfrac12 \bE_o[T] - \tfrac12 \to \infty , $$
where $T$ is the self intersection time of a non-backtracking random walk on $G_n$

Proposition \ref{prop: intersection} tells us that $L_n(x) +1 \leq Z_{n}(x)$, so that
$L_n(\tfrac1{d-1}) \leq \tfrac{d}{d-1} \bE_o[T] - 1$.

Since $g_n$ is the girth of $G_n$, we have $\bP_v [ \NB(g_n) = v ] \geq (d-1)^{-g_n}$.
Thus, for any $j \leq k -g_n$,
\begin{align*}
\bP [ \NB(j+g_n) = \NB(j) \ | \ \NB[0,j] ] & \geq (d-1)^{-g_n} . 
\end{align*}
Thus, for $k>g_n$,
\begin{align*}
\bP_o [ T > k ] & \leq \bP_o [ T > k-g_n ] \cdot \bP [ \NB(k) \neq \NB(k-g_n) \ | \ \NB[0,k-g_n] ]
\\
& \leq \bP_o [ T>k-g_n ] \cdot ( 1 - (d-1)^{-g_n} ) \\
& \leq \cdots \leq ( 1 - (d-1)^{-g_n} )^{\lfloor k/g_n \rfloor } .
\end{align*}
This easily shows that 
$$ \bE_o[T] \leq \frac{1}{ (1- (d-1)^{-g_n} ) \cdot (1 - (1- (d-1)^{-g_n} )^{1/g_n} ) } 
\leq (1+o(1)) \cdot g_n (d-1)^{g_n}   , $$
(using the inequalities $1-\xi \leq e^{-\xi}$ and $1-\tfrac12\xi \geq e^{-\xi }$, valid $0< \xi < \tfrac12$).
%
%By the proof of Corollary \ref{cor: pk lower bound}, specifically \eqref{eqn: pk lower bound}, 
%there taking $\beta = \tfrac{k}{|G_n|}$,  
%we have that there exists a constant $c>0$ such that 
%$$ \bP_o [ T>k ] \geq c\exp \sr{ -k \cdot \sr{ \tfrac1{2 \tau_n} + \tfrac{3(k+1)}{|G_n| } } } , $$
%where $\tau_n$ is the mixing time of the non-backtracking random walk.
%If we take $k = \lfloor \min \{ \tau_n , \sqrt{|G_n|} \} \rfloor$ 
%we get that for some universal constant $c'>0$,
%$\bP_o [ T>k ] \geq c' $, which implies that 
%$ \bE_o [ T ] \geq c' k . $
\end{proof}

\begin{rem}
By the proof of Corollary \ref{cor: pk lower bound}, 
as long as $\tau_n = o( (d-1)^{g_n/4} )$ as $n \to \infty$, 
taking $\beta = \tfrac{k}{|G_n|}$ in \eqref{eqn: pk lower bound},
we have that there exists a constant $c>0$ such that 
$$ \bP_o [ T>k ] \geq c\exp \sr{ -k \cdot \sr{ \tfrac1{2 \tau_n} + \tfrac{3(k+1)}{|G_n| } } } , $$
where $\tau_n$ is the mixing time of the non-backtracking random walk.
If we take $k = \lfloor \min \{ \tau_n , \sqrt{|G_n|} \} \rfloor$ 
we get that for some universal constant $c'>0$,
$\bP_o [ T>k ] \geq c' $, which implies that 
$ \bE_o [ T ] \geq c' \min \{ \tau_n , \sqrt{|G_n| } \} $ in this case.
\end{rem}

\subsection{Proofs of theorems for large girth graphs}

\begin{proof}[Proof of Theorem \ref{thm: large girth}]
This is just a combination of Lemmas \ref{lem: sub}, \ref{lem: super critical} and \ref{lem: critical}.
\end{proof}

Finally we prove Theorem \ref{thm: gamma}, which computes the critical exponent $\gamma$
for large girth graphs.

\begin{proof}[Proof of Theorem \ref{thm: gamma}]
Let $T$ be the self intersection time of a non-backtracking random walk on $G_n$.
Note that since $T \geq g_n$ a.s., which converges to infinity, 
we get using \eqref{eqn: partition function at 1} and \eqref{eqn: length at 1}
$Z_{G_n}(\tfrac{1}{d-1}) = \tfrac{d}{d-1} \bE_o [T] \to \infty , $
and 
$$ L(\tfrac1{d-1},G_n) = \frac{\bE_o[T^2] - \bE_o[T] }{2 \bE_o[T] } 
\geq \tfrac12 \bE_o[T] - \tfrac12 \to \infty . $$
So for the critical sequence $x_n = \tfrac1{d-1}$,
\begin{align*}
 \lim_{n \to \infty} \frac{ \log Z_{G_n}(x_n) }{ \log ( L(x_n,G_n)+1) } 
& \leq \lim_{n \to \infty} \frac{ \log \tfrac{d}{d-1} + \log \bE_o [T] }{ \log \bE_o[T] + \log (\bE_o[T]+1) - \log \bE_o[T] - \log 2 }  = 1.
\end{align*}
The inequality $\gamma \geq 1$ is immediate from Proposition \ref{prop: intersection}, since 
$L(x_n,G_n) +1 \leq Z_{G_n}(x_n)$.

%Now, if $x_n \to \tfrac{1}{d-1}$ is any critical sequence, then since
%$$ \lim_{n \to \infty} \frac{ \log Z_{G_n}(x_n) }{ \log Z_{G_n}(\tfrac1{d-1}) } = \lim_{n \to \infty}
%\frac{ \log (L(x_n,G_n)+1) }{ \log (L(\tfrac1{d-1}, G_n) +1) } = 1 , $$
%we have
%$$ \lim_{n \to \infty} \gamma(x_n,G_n) = 1 . $$
%
%\note{ }

The sub-critical case $x_n \to x < \tfrac{1}{d-1}$ just follows from plugging in the values 
of $L,Z$ from \eqref{eqn: partition function} and \eqref{eqn: length}.

For the super-critical case $x_n \to x > \tfrac{1}{d-1}$,
as long as $\tau_n = o((d-1)^{g_n/4})$ we can duplicate the 
super-critical case from the proof of Theorem \ref{thm: gamma MF}:
Set $z_n = \tfrac12 \cdot (\tfrac{1}{d-1} + x_n)$, so that $z_n \to z : = \tfrac12 (\tfrac1{d-1} + x) > \tfrac1{d-1}$.
By Lemma \ref{lem: super critical}, for any $w$ we have  
$$ L(z_n,G_n) \geq c (1 \mn \log ((d-1)z_n) ) \cdot |G_n| \geq c (1 \mn \log ((d-1)z_n) ) \cdot L(w,G_n) . $$
Thus,
$$ \lim_{n \to \infty} \frac{L(z_n,G_n) }{\log ( L(z_n,G_n) +1 ) } = \infty \AND 
\lim_{n \to \infty} \frac{ \log L(x_n,G_n) }{ \log L(z_n,G_n) } = 1  . $$

Consider $f(x) : = \log Z_{G_n}(x)$.  Note that $x f'(x) = L(x,G_n)$. So we have that for $x>z$, there exists $w \in [z,x]$ 
such that  
$$ \log Z_{G_n}(x) - \log Z_{G_n}(z) = f'(w) (x-z) = \tfrac{1}{w} L(w,G_n) (x-z) \geq \tfrac{x-z}{x} \cdot L(z,G_n) . $$
Thus, for $x_n,z_n$ as above,
\begin{align*}
\lim_{n \to \infty} \frac{ \log Z_{G_n}(x_n) }{ \log (L(x_n,G_n)+1) } & \geq 
1+ \lim_{n \to \infty} \frac{  (x_n-z_n) \cdot L(z_n,G_n) }{ x_n \log (L(z_n,G_n) + 1) } = \infty . 
\end{align*}
\end{proof}

\appendix

\section{Complete graph proofs}

\subsection{Mean field super-critical regime}

First a classical large deviations argument, which we include for completeness.

\begin{prop}
\label{prop: Poisson LD}
Let $P \sim \mathrm{Poi}(x^{-1})$.  Then, for all $n > x^{-1}$,
$$ \Pr [ P \geq n ] \leq  ( xn )^{-n} \cdot e^{n -\tfrac1x}   . $$
\end{prop}

\begin{proof}
For any $\lambda > 0$, the Laplace transform of $P$ is
$$ \E [ e^{\lambda P} ] = \exp \sr{ x^{-1} \cdot ( e^{\lambda} - 1 ) } . $$
Thus, for any $\lambda > 0$,
$$ \Pr [ P \geq n ] \leq \E [ e^{\lambda P} ] e^{-\lambda n }
= \exp \sr{ x^{-1} \cdot ( e^{\lambda} - 1 ) - \lambda n } . $$
Minimizing the right hand side above over $\lambda$, we obtain
$\lambda = \log (xn)$ (which is good since we assumed $xn>1$),
so
$$ \Pr [ P \geq n ] \leq \sr{ \tfrac{e}{xn} }^n \cdot e^{-1/x} . $$
\end{proof}

\begin{lem}
\label{lem: super-critical MF}
For $\eps >0$ and $x = \frac{1+\eps}{n-1}$,
we have 
$$  \abs{  L(x,K_n) - \frac{\eps}{1+\eps} (n-1) } \leq
\frac{ n I(\eps)^{n/2} }{1- I(\eps)^n } , $$
where $I(\eps) = \max \set{ e^{-\eps^2/8} , \sqrt{e}/2 }$.
Specifically, if $(\eps_n)_n$ is a sequence converging to $\Ee \in [0, \infty]$ such that 
$n \cdot \eps_n^2 \to \infty$ as $n \to \infty$, then for the sequence $x_n = \frac{1+\eps_n}{n}$,
$$ \lim_{n \to \infty} \frac{ L(x_n,K_n) }{n-1} = \frac{\Ee}{1+\Ee} , $$
where $\frac{\infty}{1+ \infty } = 1$.
\end{lem}

\begin{proof}
Let $P \sim \mathrm{Poi}(x^{-1})$, and let $Q \sim \mathrm{Poi}( \tfrac{n}{1+\eps} )$.
So $P$ is stochastically dominated by $Q$.
Thus, by Proposition \ref{prop: Poisson LD},
\begin{align*}
\Pr [ P \geq n ] & \leq \Pr [ Q \geq n ] \leq (1+\eps)^{-n} e^n e^{-\tfrac{n}{1+\eps} } 
%& 
= \exp \sr{ \tfrac{\eps}{1+\eps} n } \cdot (1+\eps)^{-n } . 
\end{align*}
We use the inequality
$ e^{\xi} \leq 1 + \xi + \tfrac{\xi^2}{2} e^{\xi}$, so if $\xi <1$,
then $e^{\xi} \leq \frac{ 1 + \xi}{1 - \xi^2 /2 } $.
Plugging in $\xi = \frac{\eps}{1+\eps}$, we obtain that when $\eps < 1$,
$$ \exp \sr{ \tfrac{\eps}{1+\eps} } (1+\eps)^{-1} \leq \frac{1 + 2\eps}{ 1 + 2 \eps + \eps^2 / 2} \leq 1 - \tfrac{\eps^2}{8} , $$
and if $\eps \geq 1$ then since $\xi e^{-\xi}$ increases when $0<\xi < 1$, with $\xi = \frac{1}{1+\eps}$ we get
$$ \exp \sr{ \tfrac{\eps}{1+\eps} } (1+\eps)^{-1} \leq e \cdot \xi e^{-\xi} \leq \frac{\sqrt{e}}{2} . $$
So we get that $ \Pr [ P \geq n ] \leq I(\eps)^n . $
Thus, for some $c = c(\eps)$,
$$ \E [ P | P \leq n-1 ] \leq \frac{\E[P]}{ 1 - \Pr [ P \geq n ] } \leq \frac{x^{-1} }{ 1 -  I(\eps)^n  }
\leq \frac{n-1}{1+\eps} \cdot \sr{ 1 + \tfrac{I(\eps)^n }{1-I(\eps)^n } } . $$

We use this bound with \eqref{eqn: Poisson MF} to obtain
that for $x \geq \frac{1+\eps}{n-1}$,
\begin{align*}
L(x,K_n)  & = n-1 - \E [ P | P \leq n-1 ] \geq (n-1) \cdot 
\sr{ 1 - \tfrac{1}{1+\eps} \cdot \sr{ 1 + \tfrac{I(\eps)^n }{1-I(\eps)^n } } } \\
& \geq \frac{ \eps (n-1)}{1+ \eps} - \frac{n I(\eps)^n }{1-I(\eps)^n }  .
\end{align*}
This proves the lower bound on $L(x,K_n)$.

For the upper bound, note that by Cauchy-Schwarz
\begin{align*}
\E [ P \1{ P \geq n} ] & \leq \sqrt{ \E[P^2]  \cdot \Pr [ P \geq n ] } \leq \sqrt{ x^{-2} + x^{-1} } \cdot I(\eps)^{n/2} \\
& \leq  n \cdot I(\eps)^{n/2} .
\end{align*}
Thus, 
\begin{align*}
L(x,K_n) & = n-1 - \E [ P | P \leq n-1 ] = n-1 - \frac{ \E[P ]  - \E[P \1{P \geq n } ] }{ 1 - \Pr [ P \geq n ] } \\
& \leq n-1 - \E[P]  + \frac{ \E[ P \1{ P \geq n } ] }{ 1 - \Pr [ P \geq n ] }
\leq \frac{\eps (n-1) }{1+ \eps} + \frac{n I(\eps)^{n/2} }{ 1-I(\eps)^n } . 
\end{align*}
(In the second inequality is where we use that $\E[P] = x^{-1} = \frac{n-1}{1+\eps}$.)
\end{proof}

\begin{rem}
Note that Lemma \ref{lem: super-critical MF} gives more than required for the super-critical phase  in 
Theorem \ref{thm: mean field}.
For $\eps_n >> n^{-1/2}$, 
we have that the expected length $L(x,K_n)$ 
is very near $\frac{\eps_n (n-1)}{1+ \eps_n}$ as $n \to \infty$.
\end{rem}

\subsection{Mean field critical regime}

\begin{lem}
\label{lem: critical MF}
There exists a constant $\alpha > 0$ such that for $x = x_n = \frac1{n-1}$,
$$ \lim_{n \to \infty} \frac{L(x_n,K_n) }{\sqrt{n-1} } = \alpha . $$
\end{lem}

\begin{proof}
Let $(P_k)_k$ be i.i.d.\ Poisson-$1$ random variables, and let $P = \sum_{k=1}^{n-1} P_k$.
So $P \sim \mathrm{Poi}(x^{-1})$.
Since $\E[P] = n-1$ and $\Var[P] = n-1$ we have by the central limit theorem that the sequence
$X_n : = \frac{P - (n-1)}{\sqrt{n-1} }$ converges in distribution to a standard Gaussian $N \sim \mathcal{N}(0,1)$.

From this, a simple application of the Portmanteau Theorem gives that 
$$ \frac{\E [ P | P \leq n-1 ] - (n-1) }{\sqrt{n-1} } = \E [ X_n | X_n \leq 0 ] \to \E [ N | N \leq 0 ] . $$
So setting $\alpha : = -\E[N | N \leq 0] = \E |N| > 0$ we have that when $x = \frac1{n-1}$,
$$ \tfrac{1}{\sqrt{n-1}} \cdot L(x_n,K_n) \to \alpha . $$
\end{proof}

\subsection{Mean field sub-critical regime}

\begin{lem}
\label{lem: sub-critical MF}
For any $\eps , \delta>0$ there exists $n_0>0$ such that for all $n>n_0$, 
if $x = \frac{1-\eps}{n-1}$ then,
$$ \frac{1-\eps}{\eps} \cdot e^{-\delta} \cdot \sr{ 1 - \tfrac3n } \leq L(x,K_n) \leq \frac{1-\eps}{ \eps } . $$
Consequently, if $(x_n)_n$ is such that $\lim_{n\to \infty} (n x_n) = 1-\eps$ then
$$ \lim_{n \to \infty} L(x_n,K_n) = \frac{1-\eps}{\eps} . $$
\end{lem}

\begin{proof}
For the upper bound we use \eqref{eqn: length} to deduce that 
$$ L(x,K_n) \leq \frac{x(n-2)}{1 - x(n-2) } \leq \frac{\eps}{1-\eps} . $$

For the lower bound, 
by \eqref{eqn: partition function}, 
\begin{align*}
Z_{K_n}(x) & \leq \frac{n-1}{n-2} \cdot \frac{1}{1- x(n-2) } \leq \frac{n-1}{n-2} \cdot \frac{1}{\eps} .
\end{align*}
Note that $1-\xi \geq e^{-2 \xi}$ for $\xi \leq \tfrac14$, so for any fixed $\delta > 0$,
\begin{align*}
L(x,K_n) & = Z_{K_n}(x)^{-1} \cdot \sum_{k=0}^{n-1} k x^k |\SAW_k(K_n)| 
%\\ & 
\geq \eps \cdot \tfrac{n-2}{n-1} \cdot \sum_{k \leq (n-1)/4} k (1-\eps)^k \prod_{j=0}^{k-1} \sr{ 1 - \tfrac{j}{n-1} } \\
& \geq \eps \cdot \tfrac{n-2}{n-1} \cdot \sum_{k \leq (n-1)/4 } k (1-\eps)^k \exp \sr{ - \tfrac{k^2}{n-1} } 
%\\ & 
\geq \tfrac{n-2}{n-1} \cdot e^{-\delta} \cdot \sum_{k \leq \sqrt{\delta(n-1)} } k (1-\eps)^k \eps \\
& \geq e^{-\delta} \cdot \tfrac{n-2}{ n-1} \cdot \frac{1-\eps}{\eps} \cdot \sr{ 1 - (1-\eps)^{\sqrt{\delta(n-1)} } } . 
\end{align*}

Taking $n \to \infty$ we have that for all $\delta>0$,
$$ \lim_{n \to \infty} L(\tfrac{1-\eps}{n-1} , K_n) \geq e^{-\delta } \cdot \frac{1-\eps}{\eps}  , $$
so  if $(x_n)_n$ is such that $\lim_{n\to \infty} (n x_n) = 1-\eps$ then
$$ \lim_{n \to \infty} L(x_n,K_n) = \frac{1-\eps}{\eps} . $$
\end{proof}

%%%%%%%%%%%%%%%%%%%%%%
%%%%%%%%%%%%%%%%%%%%%%
%%%%%%%%%%%%%%%%%%%%%%
\subsection{Proofs of theorems for mean field case}

\begin{proof}[Proof of Theorem \ref{thm: mean field}]
Combining of Lemmas 
%\ref{lem: critical seq MF},
\ref{lem: super-critical MF}, \ref{lem: critical MF} and \ref{lem: sub-critical MF} we have the asymptotics 
for the expected length $L$.

By Proposition \ref{prop: intersection}, we get that for any $\eps>0$,
$$ Z_{K_n}(\tfrac{1+\eps}{n} ) \geq L(\tfrac{1+\eps}{n}, K_n) +1 \to \infty , $$
by Lemma \ref{lem: super-critical MF},
and
$$ Z_{K_n}(\tfrac{1-\eps}{n}) \leq \tfrac{1}{1- (1-\eps) \tfrac{n-2}{n} } \to \tfrac{1}{\eps} < \infty , $$
by \eqref{eqn: partition function}.

This implies that $(\tfrac1n)_n$ is a critical sequence, and that in the super-critical case $I(x_n,G_n) \to 0$.

Also, in the sub-critical case where $n x_n \to 1-\eps$, we have that 
%a more refined consideration of \eqref{eqn: partition function} gives for $x< \tfrac{1}{n-2}$,
%$$ \bE_o [ 1 - (x(n-2))^T ]  
%\leq Z_{K_n}(x) \cdot \frac{n-2}{n-1} \cdot (1-x(n-2) )  \leq 1 . $$
$$ I(x_n,K_n) = \frac{ L(x_n,K_n)+1 }{ Z_{K_n}(x_n) } \geq  (L(x_n,K_n)+1) \cdot \sr{ 1 - (1-\eps) \tfrac{n-2}{n} }
\to 1 . $$ 
%For the critical case $x_n = \tfrac1{n-2}$, 
%we may use \eqref{eqn: partition function at 1} and \eqref{eqn: length at 1} to deduce that
%$$ I(x_n,K_n) = \frac{ L(x_n,K_n) + 1 }{ Z_{K_n}(x_n) } =  \frac{ (n-2) \bE_o [ T^2 + T ] }{ 2 (n-1) (\bE_o[T])^2 } , $$
%where $T$ is the self-intersection time of a non-backtracking walk on $K_n$.
%So the proof is complete once we establish that $\frac{\bE_o[T^2] }{ (\bE_o[T])^2 } \to 1$.
%The inequality $\bE_o[T^2] \geq (\bE_o[T])^2$ is basic, so we only provide the upper bound.
%
%\note{ !!! }
%
%-
\end{proof}

We now prove Theorem \ref{thm: gamma MF}, calculating $\gamma$ 
for sequences in the case where $G_n = K_n$, the complete graph on $n$ vertices.

\begin{proof}[Proof of Theorem \ref{thm: gamma MF}]
We start with the critical case $x_n = \tfrac{1}{n-2}$. 
In this case we have by \eqref{eqn: partition function at 1} and \eqref{eqn: length at 1} that 
$$  \tfrac{n-2}{2(n-1)}  \cdot Z_{G_n}(x_n) = \tfrac12 \bE_o[T] \leq \frac{ \bE_o[T^2 + T] }{2 \bE_o[T] } =
L(x_n,G_n) +1 \leq Z_{G_n}(x_n) . $$
Since $\bE_o [T] \to \infty$,
$$  \lim_{n \to \infty} \frac{ \log Z_{G_n}(x_n) }{ \log (L(x_n,G_n)  +1) } = 1 . $$

In the sub-critical case where $(n x_n) \to 1-\eps$, we have by \eqref{eqn: partition function}
$$ Z_{G_n}(x_n) \leq \tfrac{n-1}{(n-2) (1 - (n-2)x_n ) } \to \tfrac{1}{\eps}  , $$
and $L(x_n,G_n) +1 \to \frac1\eps$, so in this case as well,
$$ \lim_{n \to \infty} \frac{ \log Z_{G_n}(x_n) }{ \log (L(x_n,G_n)  +1) } = 1 . $$

Finally, in the super-critical case $(n x_n) \to 1+\eps$, 
set $z_n = \tfrac12 \cdot (\tfrac{1}{n} + x_n)$. 
Then $n z_n \to z : = 1+\tfrac12\eps$,
and we have by Lemma \ref{lem: super-critical MF}
that as $n \to \infty$, 
$$ L(z_n,G_n) \geq (1-o(1) ) \tfrac{\eps}{2+\eps} \cdot n \geq (1-o(1) ) \tfrac{\eps}{2+\eps}  \cdot L(x_n,G_n) . $$
Thus,
$$ \lim_{n \to \infty} \frac{L(z_n,G_n) }{\log ( L(z_n,G_n) +1 ) } = \infty \AND 
\lim_{n \to \infty} \frac{ \log L(x_n,G_n) }{ \log L(z_n,G_n) } = 1  . $$

Consider $f(x) : = \log Z_{G_n}(x)$.  Note that $x f'(x) = L(x,G_n)$. So we have that for $x>z$, there exists $w \in [z,x]$ 
such that  
$$ \log Z_{G_n}(x) - \log Z_{G_n}(z) = f'(w) (x-z) = \tfrac{1}{w} L(w,G_n) (x-z) \geq \tfrac{x-z}{x} \cdot L(z,G_n) . $$
Thus, for $x_n,z_n$ as above,
\begin{align*}
 \lim_{n \to \infty} & \frac{ \log Z_{G_n}(x_n) }{ \log (L(x_n,G_n)+1) } \\
& \geq \lim_{n \to \infty} \frac{ \log Z_{G_n}(z_n) }{ \log (L(z_n,G_n)+1) } 
+ \frac{ (x_n-z_n) \cdot L(z_n,G_n) }{ z_n \cdot \log (L(z_n,G_n)+1 )} \\
& \geq 1 + \tfrac{\eps}{2(1+\eps) } \cdot \lim_{n \to \infty} \frac{ L(z_n,G_n) }{ \log (L(z_n,G_n)+1) } = \infty . 
\end{align*}
\end{proof}


\begin{thebibliography}{----}



\bibitem{AlmJanson}
S.E.~Alm and S.~Janson. 
Random self-avoiding walks on one-dimensional lattices.
{\em Commun. Statist.-Stochastic Models}  {\bf 6}(2), (1990) 169--212.




\bibitem{ABLS07}
N.~Alon, I.~Benjamini, E.~Lubetzky and S.~Sodin. 
Non-backtracking random walks mix faster. 
{\em Communications in Contemporary Mathematics} {\bf 9} (2007), 585--603.


\bibitem%[BDCGS12]
{hugo}
R.~ Bauerschmidt, H.~ Duminil-Copin, J.~ Goodman and G.~ Slade.
Lectures on self-avoiding-walks. in: {\em Probability and Statistical Physics in Two and More Dimensions}. Eds: D.~ Ellwood, C.~ Newman, V.~ Sidoravicius, W.~ Werner, {\em Clay Mathematics Proceedings}, {\bf vol. 15}, Amer. Math. Soc., Providence, RI, 2012, 395--467. 
%available at: 
\href{http://arxiv.org/abs/1109.1549}{\nolinkurl{arXiv:1109.1549}}

 \bibitem%[BDS]
{BDS11}
D.C.~ Brydges, A.~Dahlqvist and G.~Slade. 
The strong interaction limit of continuous-time weakly self-avoiding walk. %available at:
  \href{http://arxiv.org/abs/1104.3731}{\nolinkurl{arXiv:1104.3731}}

\bibitem{BEI92}
D.C.~Brydges, S.N.~Evans, and J.Z.~Imbrie. 
Self-avoiding walk on a hierarchical lattice in four dimensions. 
{\em Ann. Probab.}  {\bf 20} (1992), 82--124.

\bibitem{BI03A}
D.C.~Brydges and J.Z.~Imbrie. 
End-to-end distance from the GreenÕs function for a hierarchical self-avoiding walk in four dimensions. 
{\em Commun. Math. Phys.} {\bf 239} (2003), 523--547.

\bibitem{BI03B}
D.C.~Brydges and J.Z.~Imbrie. 
GreenÕs function for a hierarchical self-avoiding walk in four dimensions. 
{\em Commun. Math. Phys.} {\bf 239} (2003), 549--584.


\bibitem%[BIS09]
{BrydgesImbrieSlade}
D.C.~Brydges, J.Z.~Imbrie and G.~Slade.
Functional integral
  representations for self-avoiding walk.
  {\em Probab. Surv.} \textbf{6} (2009),
  34--61.
   %available at:
  %\href{http://www.i-journals.org/ps/viewarticle.php?id=152&layout=abstract%
%}{i-journals.org}.

\bibitem%[BS]
{BS10}
D.C.~Brydges and G.~Slade.
Renormalisation group analysis of weakly
  self-avoiding walk in dimensions four and higher. {\em lecture at ICM 2010,
  Hyderabad}. %available at:
  \href{http://arxiv.org/abs/1003.4484}{\nolinkurl{arXiv:1003.4484}}

\bibitem%[BS85]
{BS85}
D.C.~Brydges and T.~Spencer.
Self-avoiding walk in $5$ or more
  dimensions. 
  {\em Commun. Math. Phys.} \textbf{97}(1-2) (1985), 125--148.
 %available at:
  %\href{http://projecteuclid.org/DPubS?service=UI&version=1.0&verb=Display&handle=euclid.cmp/1103941982}{projecteuclid.org}.

%\bibitem%[DCH12]
%{DH12}
%H.~Duminil-Copin and A.~Hammond, 
%Self-avoiding walk is sub-ballistic. 
%available at: \href{http://arxiv.org/abs/1205.0401}{arXiv:1205.0401}.

\bibitem{HugoAlan}
H.~Duminil-Copin and A.~Hammond. 
Self-avoiding walk is sub-ballistic. 
{\em Communications in Mathematical Physics} {\bf 324} (2013), 401--423.



\bibitem{DKY12}
H.~Duminil-Copin, G.~Kozma and A.~Yadin.
Supercritical self-avoiding walks are space-filling.
{\em 	Annales de l`IHP Probab. et Stat.} to appear.
%Available at: 
\href{http://arxiv.org/abs/1110.3074}{\nolinkurl{arXiv:1110.3074}}


%\bibitem%[DCS12]
%{DS12}
%H.~Duminil-Copin and S.~Smirnov. 
%The connective constant of the
%  honeycomb lattice equals $\sqrt{2+\sqrt 2}$.
%  {\em Annals of Mathematics} \textbf{175}(3) (2012), 1653--1665. 


\bibitem{Fitzner}
{\sc R.~Fitzner}.
{\em Non-backtracking lace expansion.}
Ph.D.\ thesis, TU Eindhoven, 2013.

%\bibitem%[Flo53]
%{Flory}
%P.~Flory, \emph{Principles of polymer chemistry}, Cornell University Press,
%  1953.

%\bibitem%[Gri99]
%{GrimmettPerco}
%G.~Grimmett, \emph{Percolation}, Springer Verlag, 1999.




\bibitem{GZ13}
{G.R.~Grimmett and Z.~Li}.
Counting self-avoiding walks.
\href{http://arxiv.org/abs/1304.7216}{\nolinkurl{arXiv:1304.7216}}

\bibitem{Hara08}
T.~Hara.
Decay of correlations in nearest-neighbor self-avoiding walk, percolation, lattice trees and animals. 
{\em Ann. Probab.} {\bf 36} (2008), 530--593.


\bibitem%[HS91]
{HaraSlade1}
T.~Hara and G.~Slade.
Critical behaviour of self-avoiding walk in five or
  more dimensions. 
  {\em Bull. Amer. Math. Soc. (N.S.)} \textbf{25}(2) (1991),
  417--423. 
  %available at:
 % \href{http://www.ams.org/journals/bull/1991-25-02/S0273-0979-1991-16085-4%
%/home.html}{ams.org}.

\bibitem%[HS92]
{HaraSlade2}
%\bysame, 
T.~Hara and G.~Slade.
Self-avoiding walk in five or more dimensions. {I}. {T}he
  critical behaviour. 
  {\em Commun. Math. Phys.} \textbf{147}(1) (1992), 101--136.
%  available:
 % \href{http://projecteuclid.org/DPubS?service=UI&version=1.0&verb=Display&%
%handle=euclid.cmp/1104250528}{projecteuclid.org}.


\bibitem%[HS92]
{HaraSlade3}
T.~Hara and G.~Slade.
The lace expansion for self-avoiding walk in five or more dimensions. 
{\em Reviews in Math. Phys.} {\bf 4} (1992), 235--327.


%\bibitem%[HW62]
%{HammersleyWelsh}
%J.~M.~ Hammersley and D.~J.~A.~ Welsh, 
%Further results on the rate of
%  convergence to the connective constant of the hypercubical lattice, 
%  {\em Quart.
%  J. Math. Oxford Ser. (2)} \textbf{13} (1962), 108--110. 
%  available at;
%  \href{http://qjmath.oxfordjournals.org/content/13/1/108.full.pdf+html}{ox%
%fordjournals.org}.

\bibitem%[Iof98]
{Ioffe}
D.~Ioffe. 
Ornstein-Zernike behaviour and analyticity of shapes for
  self-avoiding walks on {${\bf Z}^d$}. 
  {\em Markov Process. Related Fields}
  \textbf{4}(3) (1998), 323--350.


\bibitem{Kesten}
H.~Kesten.
On the number of self?avoiding walks II. 
{\em Journal of Mathematical Physics} {\bf 5}(8) (1964), 1128--1137.


\bibitem%[LSW04]
{LSW5}
G.F.~ Lawler, O.~Schramm and W.~Werner. 
On the scaling limit of planar
  self-avoiding walk, in: {\em Fractal geometry and applications: a jubilee of
  Beno\^\i t Mandelbrot}, Part 2, Proc. Sympos. Pure Math., vol.~72, Amer.
  Math. Soc., Providence, RI, 2004, pp.~339--364. %Available at:
  \href{http://arxiv.org/abs/math/0204277}{\nolinkurl{arXiv:math/0204277}}

\bibitem{LPS88}
A.~Lubotzky, R.~Phillips and P.~Sarnak. 
Ramanujan graphs. {\em Combinatorica} {\bf 8} (1988), 261--277.


\bibitem%[MS93]
{MadrasSlade}
N.~Madras and G.~Slade. 
{\em The self-avoiding walk}. Probability and its
  Applications, Birkh\"auser Boston Inc., Boston, MA, 1993.


\bibitem{MadrasWu}
N.~Madras and C.C.~Wu. 
Self-avoiding walks on hyperbolic graphs. 
{\em Combinatorics Probability and Computing} {\bf 14} (2005), p.~523.


\bibitem{Mar82}
G.A.~Margulis. Explicit constructions of graphs without short cycles and low density codes. 
{\em Combinatorica} {\bf 2} (1982), 71--78.


\bibitem{Nach09}
A.~Nachmias.
Mean-field conditions for percolation on finite graphs.
{\em GAFA} {\bf  19} (2009), 1171--1194. 

\bibitem{NP12}
A.~Nachmias and Y.~Peres. 
Non-amenable Cayley graphs of high girth have $p_c<p_u$ and mean-field exponents. 
{\em Electronic Communications in Probability} {\bf 17} (2012), 1--8.


\end{thebibliography}
\end{document}